\documentclass[a4paper]{amsart}

\usepackage[T1]{fontenc}
\usepackage[latin1]{inputenc}
\usepackage{amsfonts}
\usepackage{amssymb}
\usepackage{amsxtra}
\usepackage{amsthm}
\usepackage{ae}
\usepackage[all]{xy}

\newcommand{\sg}{\sigma}
\newcommand{\tsg}{\tilde \sigma}
\renewcommand{\a}{\alpha}
\renewcommand{\b}{\beta}
\newcommand{\e}{\varepsilon}
\renewcommand{\l}{\lambda}
\renewcommand{\d}{\delta}
\renewcommand{\t}{\theta}
\newcommand{\p}{\phi}
\newcommand{\vp}{\varphi}
\newcommand{\Th}{\Theta}
\newcommand{\oc}{c^\ast}
\newcommand{\half}{\frac{1}{2}}
\newcommand{\sesq}{\frac{3}{2}}

\DeclareMathOperator{\sgn}{sgn}
\DeclareMathOperator{\ind}{Index}
\DeclareMathOperator{\tr}{Trace}
\DeclareMathOperator{\res}{Res}

\newcommand{\I}{\mathcal{I}}
\newcommand{\D}{\mathcal{D}}
\newcommand{\AD}{\mid \mathcal{D} \mid}
\renewcommand{\H}{\mathcal{H}}
\renewcommand{\S}{\mathcal{S}}

\newcommand{\A}{\mathcal{A}}
\newcommand{\B}{\mathcal{B}}
\newcommand{\Br}{\mathfrak{B}}
\newcommand{\NH}{\mathbb{N}}
\newcommand{\ZH}{\mathbb{Z}}
\newcommand{\CH}{\mathbb{C}}
\newcommand{\RH}{\mathbb{R}}
\renewcommand{\TH}{\mathbb{T}}
\newcommand{\G}{\mathfrak{G}}
\renewcommand{\Re}{\text{Re}}
\renewcommand{\Im}{\text{Im}}

\newcommand{\lcom}{\left[}
\newcommand{\rcom}{\right]}
\newcommand{\ot}{\otimes}
\newcommand{\set}[1]{\lbrace #1 \rbrace}
\newcommand{\wh}{\widehat}
\newcommand{\fl}{\text{for all }}
\newcommand{\ket}[1]{|#1\rangle}    
\newcommand{\kett}[1]{|#1\rangle\!\rangle} 
\newcommand{\ov}[1]{\overline{#1}}

\def\build#1_#2^#3{\mathrel{
\mathop{\kern 0pt#1}\limits_{#2}^{#3}}}

\newbox\ncintdbox \newbox\ncinttbox
\setbox0=\hbox{$-$}
\setbox2=\hbox{$\displaystyle\int$}
\setbox\ncintdbox=\hbox{\rlap{\hbox
    to \wd2{\kern-.1em\box2\relax\hfil}}\box0\kern.1em}
\setbox0=\hbox{$\vcenter{\hrule width 4pt}$}
\setbox2=\hbox{$\textstyle\int$}
\setbox\ncinttbox=\hbox{\rlap{\hbox
    to \wd2{\kern-.05em\box2\relax\hfil}}\box0\kern.1em}
\newcommand{\ncint}{\mathop{\mathchoice{\copy\ncintdbox}%
    {\copy\ncinttbox}{\copy\ncinttbox}%
    {\copy\ncinttbox}}\nolimits}

\newcommand{\up}{\uparrow}
\newcommand{\ra}{\rightarrow}
\newcommand{\dn}{\downarrow}
\newcommand{\lgans}{\textquotedblleft}
\newcommand{\rgans}{\textquotedblright}
\newcommand{\nn}{\nonumber}

\numberwithin{equation}{section}
\theoremstyle{plain}
\newtheorem{thm}{Theorem}[section]
\newtheorem{lem}[thm]{Lemma}
\newtheorem{prop}[thm]{Proposition}
\newtheorem{cor}[thm]{Corollary}
\newtheorem*{1_ax}{Dimension}
\newtheorem*{2_ax}{Regularity}
\newtheorem*{3_ax}{Dimension Spectrum}

\theoremstyle{definition}
\newtheorem{defn}{Definition}[section]

\theoremstyle{remark}
\newtheorem{rem}{Remark}[section]

\begin{document}

\title[Chern-Simons action]{A Chern-Simons action for noncommutative spaces in general with the example $SU_q(2)$}

\author{Oliver Pfante}

\address{Oliver Pfante \\
         Mathematisches Institut\\
         Westfälische Wilhelms-Universität Münster \\
         Einsteinstraße 62 \\
         48149 Münster\\
         Germany 
}
\email{oliver.pfante@uni-muenster.de}

\maketitle

\begin{abstract}
Witten constructed a topological quantum field theory with the Chern-Simons action as Lagrangian. We define a Chern-Simons action for $3$-dimensional spectral triples. We prove gauge invariance of the Chern-Simons action, and we prove that it concurs with the classical one in the case the spectral triple comes from a $3$-dimensional spin manifold. In contrast to the classical Chern-Simons action, or a noncommutative generalization of it introduced by A. H. Chamseddine, A. Connes, and M. Marcolli by use of cyclic cohomology, the formula of our definition contains a linear term which shifts the critical points of the action, i. e. the solutions of the corresponding variational problem.\\
Additionally, we investigate and compute the action for a particular example: the quantum group $SU_q(2)$. Two different spectral triples were constructed for $SU_q(2)$. We investigate the Chern-Simons action, defined in the present paper, in both cases, and conclude the non-topological nature of the action. Using the Chern-Simons action as Lagrangian we define and compute the path integral, at least conceptually.
\end{abstract}

\section{Introduction}

S.-S. Chern and J. Simons uncovered a geometrical invariant which grew out of an attempt to derive a purely combinatorial formula for the first Pontrjagin number of a $4$-manifold. This Chern-Simons invariant turned out to be a higher dimensional analogue of the $1$-dimensional geodesic curvature and seemed to be interesting in its own right \cite{chern_simons}. This classical Chern-Simons invariant has found numerous applications in differential geometry, global analysis, topology and theoretical physics. \\
E. Witten \cite{witten} used the Chern-Simons invariant to derive a $3$-dimensional quantum field theory in order to give an intrinsic definition of the Jones Polynomial and its generalizations dealing with the mysteries of knots in three dimensional space. Witten's approach is motivated by the Lagrangian formulation of quantum field theory, where observables are computed by means of integration of an action functional over all possible physical states modulo gauge transformations -- i. e. one computes the Feynman \textit{path integral}. \\
Knot polynomials deal with topological invariants, and understanding these theories as quantum field theories, as Witten did, involves the construction of theories in which all of the observables are topological invariants. The physical meaning of \textquotedblleft topological invariance\textquotedblright $\,$is \textquotedblleft general covariance\textquotedblright. A quantity that can be computed from a manifold $M$ as a topological space without any choice of a metric is called a \textquotedblleft topological invariant" by mathematicians. To a physicist, a quantum field theory defined on a manifold $M$ without any a priori choice of a metric on $M$ is said to be generally covariant. Obviously, any quantity computed in a generally covariant quantum field theory will be a topological invariant. Conversely, a quantum field theory in which all observables are topological invariants can naturally be seen as a generally covariant quantum field theory. The surprise, for physicists, comes in how general covariance in Witten's three dimensional covariant quantum field theory is achieved. General relativity gives us a prototype for how to construct a quantum field theory with no a priori choice of metric - we introduce a metric, and then integrate over all metrics. Witten constructs an exactly soluble generally covariant quantum field theory in which general covariance is achieved not by integrating over metrics, instead he starts with a gauge invariant Lagrangian that does not contain any metric -- the Chern-Simons invariant, which becomes the \textit{Chern-Simons action}. \\

Let $M$ denote a $3$-dimensional, closed, oriented manifold and $A \in M_N(\Omega^1(M))$ a hermitian $N \times N$-matrix of differential $1$-forms on $M$. The Chern-Simons action for $A$ is given by:
\begin{equation} \label{eq_1.1}
S_{CS}(A) := \dfrac{k}{4 \pi} \int_M \tr \left( A \wedge dA + \dfrac{2}{3} A \wedge A \wedge A \right) 
\end{equation}
for an integer $k \in \ZH$ which is called the \textit{level}. Clearly, $S_{CS}(A)$ does not depend on a metric, and therefore it is a topological invariant. In addition, the Chern-Simons action is gauge invariant. Under gauge transformation 
\begin{equation*}
A \mapsto A^u= uAu^\ast + u d u^\ast
\end{equation*} 
of the $1$ form by a gauge-map $u: M \ra SU(N)$ we have 
\begin{equation*}
S_{CS}(A^u) = S_{CS}(A) + 2 \pi m
\end{equation*}
for an integer $m \in \ZH$, called the winding number of the gauge-map $u$. Next Witten computes the partition function 
\begin{equation} \label{eq_1.2}
Z(k) = \int  D \lcom A \rcom e^{i S_{CS}(A)} 
\end{equation} 
of the action, where we integrate over all possible hermitian $N \times N$-matrices of differential $1$-forms $A$ on $M$ modulo gauge transformation. Note that the integrand is well defined because it does not depend on the special choice of the representative $A$ of the equivalence class $\lcom  A \rcom$ due to the gauge invariance of the action.\\
The partition function $Z(k)$ is a topological invariant of the manifold $M$ in the variable $k$ corresponding to the so called \textit{no knot}-case in \cite{witten}, which is interesting in its own right. There exists a meromorphic continuation of $Z(k)$ in $k$ which was computed in \cite{witten} for some explicit examples. For the $3$-sphere $M = S^3$ and the gauge group $SU(2)$ one gets 
\begin{equation*}
Z(k) = \sqrt{\dfrac{2}{2+k}} \sin \left( \dfrac{\pi}{k + 2} \right) \, .
\end{equation*}
For manifolds of the form $X \times S^1$, for a $2$-dimensional manifold $X$, and arbitrary gauge group $SU(N)$, one obtains $Z(k) = N_X$ as result, with an integer $N_X \in \NH$ depending only on the manifold $X$.\\

Our main concern in this paper is the definition of a Chern-Simons action on noncommutative spaces, so called spectral triples $\left( \A, \H, \D \right)$, consisting of a pre-$C^\ast$-algebra $\A$ of bounded operators on the Hilbert space $\H$ and an unbounded, self-adjoint operator $\D$ on $\H$ with compact resolvent, such that the commutators $\lcom a, \D \rcom$ for $a \in \A$ are bounded. Spectral triples are the starting point for the study of noncommutative manifolds \cite{connes}. One can think of them being a generalization of the notion of ordinary differential manifolds because every spin manifold without boundary can be encoded uniquely by a spectral triple \cite{connes_reconstruction}. In this framework some core structures of topology and geometry such as the index theorem \cite{connes_index} were extended far beyond the classical scope.\\
By means of the local index theorem proven by A. Connes and H. Moscovici in \cite{connes_index} we define a Chern-Simons action for noncommutative spaces in general. This definition looks similar to the classical one \eqref{eq_1.1}, but an additional linear term appears which comes from the $1$-cochain of the cyclic cocycle given by the local index formula. Next we can reconstruct the analogue of gauge transformations and prove gauge invariance of the Chern-Simons action. We derive this most important result of this work by means of the coupling between cyclic cohomology and $K$-theory \cite{connes}. Finally, we prove that the Chern-Simons action, defined by us, coincides with the classical one if one works with a \textit{commutative} spectral triple coming from a $3$-dimensional spin manifold.\\

Our definition of a noncommutative Chern-Simons action is not the first one. Several proposals were given before: for instance by T. Krajewski \cite{krajewski} and A. H. Chamseddine, J. Fröhlich \cite{chamsedine}. The Chern-Simons action in \cite{krajewski}, which is defined as the classical one \eqref{eq_1.1} using the Wodzicki residue instead of the integral, fails to be gauge invariant in general. For the quantum group $SU_q(2)$ and its spectral triple constructed by P. S. Chakraborty and A. Pal \cite{pal} the Chern-Simons action of \cite{krajewski} is not gauge invariant in general.\\
In \cite{chamsedine} a definition of a Chern-Simons action is given, using a Chern-Simons form constructed in cyclic cohomology by D. Quillen \cite{quillen}. But the definition in \cite{chamsedine} covers only the case, where the spectral triple essentially comes from a cylindrical manifold $\lcom 0,1 \rcom \times M$, with a $2$-dimensional manifold $M$, where noncommutativity is implemented by tensoring the algebra of smooth functions on the manifold with a finite dimensional matrix-algebra.\\
The very first consideration about a Chern-Simons action in noncommutative geometry can be found in \cite{string} by E. Witten himself. In the book \cite{connes_marcolli} of A. Connes and M. Marcolli these considerations were generalised (based on a joint work \cite{connes_cham} of the first author with A. H. Chamseddine, where a Chern-Simons term appears in the variation of the spectral action under inner fluctuations) in order to define a Chern-Simons action for arbitrary unital, involutive algebras $\A$ in the framework of cyclic cohomology. Let $\psi$ be a cyclic $3$-cocycle on $\A$. The functional
\begin{equation*}
CS_{\psi}(A) = \int_{\psi} A dA + \dfrac{2}{3} A^3 \, , \quad A \in \Omega^1(\A) \, ,
\end{equation*}
with $\psi(a^0,a^1,a^2,a^3) = \int_\psi a^0da^1da^2da^3$, transforms under gauge transformation $A \mapsto A^u = u d u^\ast + u A u^\ast$, for a unitary $u \in \A$, as
\begin{equation*}
CS_\psi(A^u) = CS_\psi(A) + \dfrac{1}{3} \langle \psi, u \rangle \, ,
\end{equation*}
where $\langle \psi, u \rangle$ is the pairing between $HC^3(\A)$ and $K_1(\A)$. This construction is similar to the one which is given in this paper, but there are also some differences. \\
Firstly, our construction is more explicit. We make use of the additional structure given by the spectral triple in order to define a unique Chern-Simons action for the particular triple. The definition by means of cyclic cohomology does not suggest uniqueness, and there is the question left which cocycle should be chosen. Secondly, our definition contains an additional linear term: the cocycle $\phi_1$ of the local index formula. Such an additional linear term in the definition of a noncommutative Chern-Simons action is new but not necessarily unreasonable. In \cite{wulkenhaar_victor}, V. Gayal and R. Wulkenhaar investigate the Yang-Mills action for a triple constructed on the noncommutative d-dimensional Moyal space and discovered also an additional linear part of the Yang-Mills action for this noncommutative space.\\

In the present paper we shall apply additionally the general results about a noncommutative Chern-Simons action to a special example of noncommutative spaces: the quantum group $SU_q(2)$ with deformation parameter $ 0 \leq q < 1$. One can think of $SU_q(2)$ as a noncommutative generalization of the ordinary $3$-sphere $S^3$, which is isomorphic to $SU(2)$. Spectral triples for $SU_q(2)$ were constructed by  P. S. Chakraborty and A. Pal \cite{pal} on the one hand, and L. D\c{a}browski, G. Landi, A. Sitarz, W. van Suijlekom and J. C. V$\acute{a}$rilly \cite{Landi_1} on the other hand. The index theoretical aspects of the spectral triple in \cite{pal} were carefully studied by A. Connes \cite{connes_sphere}. The same was done for the spectral triple in \cite{Landi_1} by the same authors in \cite{Landi_2}. \\
In this article we mainly work with the spectral triple $(C^\infty(SU_q(2)), \H, \D)$ of \cite{pal} using the results in \cite{connes_sphere}. We construct a symbol map $\sg_q: C^\infty(SU_q(2)) \ra C^{\infty}(S^1)$ and show that the Chern-Simons action is essentially a pull-back by $\sg_q$ of an action functional on the algebra $C^\infty(S^1)$. Using this, we can reduce complexity of the path integral even when gauge-breaking is involved. We compute the Chern-Simons action explicitly, and obtain an additional linear term which is characteristic for our approach. This additional linear term shifts the critical points of the action and has made explicit computations of the path integral impossible so far.\\
Another triple is available for $SU_q(2)$, given by the authors of \cite{Landi_1}, which is slightly different from the one in \cite{pal}. The modifications in \cite{Landi_1} are performed in such a way that the resulting triple is similar with the commutative one, given by the $3$-sphere $SU(2)$. There is only one crucial problem with this construction: the \lgans volume form\rgans $\,$, the cochain $\phi_3$ of the local index formula, whose special shape is determined in \cite{Landi_2}, vanishes identically. The proof bases upon the results in \cite{Landi_2} and is carried out in the present paper. Hence the noncommutative analogue of a Chern-Simons action, defined in the present paper, becomes rather trivial and uninteresting. Although Chern-Simons theory is not particularly interesting for this spectral triple, it delivers an extremely useful insight into the nature of our noncommutative approach. The different shape of the Chern-Simons action for the triples in \cite{pal} or \cite{Landi_1} proves the non-topological nature of our action. This is interesting in its own right because in the classical setting, working with a $3$-dimensional manifold, the topological nature of the Chern-Simons action \eqref{eq_1.1} is rather obvious. We prove that the noncommutative action coincides with the classical one for commutative triples coming from a $3$-dimensional spin manifold. Hence one might expect topological invariance of the action defined in the present paper too. Here by topological invariance we mean that the action should depend only on the underlying \lgans noncommutative topological space\rgans, which is in both cases the $C^\ast$-algebra $C(SU_q(2))$.\\

The paper is divided into two parts. In the first part a general definition of a Chern-Simons action on noncommutative spaces is worked out thoroughly. This part starts with the second section, where we introduce the main tools of noncommutative differential geometry in order to keep this paper self-contained. In the third section we define a noncommutative Chern-Simons action. Additionally, we prove gauge invariance of this noncommutative Chern-Simons action, and we prove that this action coincides with the classical one \eqref{eq_1.1} in the case the spectral triple comes from a $3$-dimensional spin manifold.\\
In the second part of the present paper we apply the general results to $SU_q(2)$ and its $3$-dimensional spectral triples. In the fourth section we give a rough overview of the triple in \cite{pal} and its properties proven there and in \cite{connes_sphere}. Distinguishing the case $q = 0$ from $0 < q < 1$ becomes important, especially for the local index theorem, because the shape of the cyclic cocycle $( \phi_3, \phi_1)$ is different. Eliminating this distinction is the main result of a technical lemma in the fifth section, where we construct a symbol map $\sg_q: C^\infty(SU_q(2)) \ra C^{\infty}(S^1)$ for $0 \leq q < 1$. In terms of this symbol map we can express the crucial property of the Chern-Simons action: a linearity property of the action which we take into account, when we define and compute, at least conceptually, a noncommutative analogue of the path integral \eqref{eq_1.2} in the sixth section. In the last section we analyse the differences that appear if one works with the spectral triple given in \cite{Landi_1} instead, and deduce the non-topological nature of the Chern-Simons action defined in the first part.\\

This paper is one of two papers about the content of my Ph. D. thesis. An additional paper \cite{pfante_torus} will appear, where we compute the Chern-Simons action explicitly for the noncommutative $3$-torus. Also there we make sense of the path integral for this noncommutative space, i. e. we give a definition of the measure $ D \lcom A \rcom$ in \eqref{eq_1.2}. In addition, for the noncommutative $3$-torus we were able to carry out computations of the path integral more explicitly than for $SU_q(2)$, i. e. instead of a conceptual result only, in \cite{pfante_torus} the first coefficient in the Taylor expansion series of the partition function $Z(k)$ in the variable $k^{-1}$ is computed. \\

It is a great pleasure for me to thank my advisor Raimar Wulkenhaar who took care of me and my work for years. In addition, I appreciate helpful discussions with Alain Connes, Giovanni Landi, Walter van Suijlekom and Christian Voigt.

\part{Chern-Simons action on noncommutative spaces}

This part establishes the main results of the paper concerning the definition of a Chern-Simons action on $3$-dimensional spectral triples. Gauge invariance of the action is proved for a noncommutative analogue of gauge mapping, and that the action coincides with the classical Chern-Simons action if one works with a spectral triple coming from a $3$-dimensional spin manifold.

\section{Noncommutative differential geometry}

In this section we give some preliminaries about noncommutative differential geometry \cite{connes}. Starting point is the definition of a spectral triple.

\begin{defn}
A spectral triple $(\A, \H, \D)$ consists of a unital pre-$C^\ast$-algebra $\A$, a separable Hilbert space $\H$ and a densely defined, self-adjoint operator $\D$ on $\H$ such that
\begin{itemize}
 \item there is a faithful representation $\pi: \A \ra \B(\H)$ of the pre-$C^\ast$-algebra $\A$ by bounded operators on $\H$
 \item the commutator $\lcom \D, \pi(a) \rcom $ is densely defined on $\H$ for all $a \in \A$ and extends to a bounded operator on $\H$
 \item the resolvent $(\D - \l)^{-1}$ extends for $\l \notin \RH$ to a compact operator on $\H$.
\end{itemize}
\end{defn}

\begin{rem}
A pre-$C^\ast$-algebra differs from an ordinary $C^\ast$-algebra being only closed with respect to holomorphic functional calculus. 
\end{rem}

The commutative world provides us with a large class of so called commutative spectral triples $(C^{\infty}(M), L^{2}(\S), \D)$ consisting of the Dirac operator acting on the Hilbert space $L^{2}(\S)$ of $L^{2}$-spinors over a closed spin manifold $M$. \\
Next we introduce the noncommutative replacement of $n$-forms. See \cite{voigt}.

\begin{defn}
For $n \geq 0$ let $\Omega^n(\A) = \A \ot \ov \A^{\otimes n}$ the vector space of noncommutative $n$-forms over $\A$, where $\ov \A = \A - \CH 1$ denotes the vector space we obtain from the algebra $\A$ if one takes away the unital element $1 \in \A$.  
\end{defn}

Here we used the notation
\begin{equation*}
\ov \A^{\otimes n} = \ov \A \otimes \cdots \otimes \ov \A
\end{equation*}
to denote the tensor product of $n$ copies of $\ov \A$. Elements of $\Omega^n(\A)$ are written in the more suggestive form $a^0da^1 \cdots da^n$ for $a^0, \ldots, a^n \in \A$, setting $a^0da^1 \cdots da^n = 0$ if $a^i = 1$ for $i \geq 1$. We also write $da^1 \cdots da^n$ if $a^0 = 1$.

\begin{rem}
The notion of noncommutative $n$-forms and their homology theory can be built up for arbitrary algebras, especially non-unital ones. In this case the definition differs slightly from the one we have just given. See \cite{voigt} for more details.
\end{rem}

One can define a $\A$-$\A$-bimodule structure on $\Omega^n(\A)$. Let us first consider the case $n =1$. We define a left $\A$-module structure on $\Omega^{1}(\A)$ by setting 
\begin{equation*}
a \left( a^{0}da^{1} \right)  = aa^{0}da^{1}.
\end{equation*}
A right $\A$-module structure on $\Omega^{1}(\A)$ is defined according to the Leibniz rule $d(ab) = dab+adb$ by 
\begin{equation*}
\left( a^{0}da^{1} \right) a=a^{0}d \left( a^{1}a \right) -a^{0}a^{1}da \, .
\end{equation*}
With these definitions $\Omega^{1}(\A)$ becomes an $\A$-$\A$-bimodule. It is also easy to verify that there is a natural isomorphism
\begin{equation*}
\Omega^{n}(\A) \cong \Omega^{1}(\A) \otimes_{\A} \Omega^{1}(\A) \otimes_{\A} \cdots \otimes_{\A} \Omega^{1}(\A) = \Omega^{1}(\A)^{\otimes_{\A}n}
\end{equation*} (as vector spaces) for every $n \geq 1$. As a consequence, the spaces $\Omega^{n}(\A)$ are equipped with an $\A$-$\A$-bimodule structure in a natural way. Explicitly, the left $\A$-module structure on $\Omega^{n}(\A)$ is given by
\begin{equation*}
a \left( a^{0} da^{1} \cdots da^{n} \right) = aa^{0}da^{1}\cdots da^{n} 
\end{equation*} 
and the right $\A$-module structure may be written as 
\begin{align*} 
 \left( a^{0} da^{1} \cdots d a^{n} \right) a = \, & a^{0} d a^{1} \cdots d \left(  a^{n}  a \right) + \\
						& \sum^{n-1}_{j=1} (-1)^{n-j} a^{0} da^{1} \cdots d(a^{j}a^{j+1}) \cdots da^{n}da \\
						& + (-1)^{n} a^{0}a^{1}da^{2}\cdots da^{n}da.
\end{align*}
Moreover, we view $\Omega^{0}(\A) = \A$ as an $\A$-bimodule in the obvious way using the multiplication in $\A$. We are also able to define a map $\Omega ^{n}(\A) \otimes \Omega^{m}(\A) \ra \Omega^{n+m}(\A)$ by considering the natural map
\begin{equation*}
\Omega^{1}(\A)^{\otimes_{\A}n} \otimes \Omega^{1}(\A)^{\otimes_{\A}m} \rightarrow \Omega^{1}(\A)^{\otimes_{\A}n} \otimes_{\A} \Omega^{1}(\A)^{\otimes_{\A}m} = \Omega^{1}(\A)^{\otimes_{\A}n+m}.
\end{equation*}
Let us denote by $\Omega(\A)$ the direct sum of the spaces $\Omega^{n}(\A)$ for $n \geq 0$. Then the maps $\Omega^{n}(\A) \otimes \Omega^{m}(\A) \rightarrow \Omega ^{n+m}(\A) $ assemble to the map $\Omega(\A) \otimes \Omega(\A) \rightarrow \Omega(\A)$. In this way the space $\Omega(\A)$ becomes an algebra. Actually $\Omega(\A)$ is a graded algebra if one considers the natural grading given by the degree of a differential form. Let us now define a linear operator $d: \Omega^{n}(\A) \ra \Omega^{n+1}(\A)$ by 
\begin{equation*}
 d \left( a^{0}da^{1}\cdots da^{n} \right) = d a^{0} \cdots d a^{n}, \quad d \left( d a^{1} \cdots d a^{n} \right) = 0
\end{equation*}
for $a^0, \ldots, a^n \in \A$. It follows immediately from the definition that $d^{2}= 0$. A differential form in $\Omega(\A)$ is called homogeneous of degree $n$ if it is contained in the subspace $\Omega^{n}(\A)$. Then the graded Leibniz rule
\begin{equation*}
d(\omega \eta) = d\omega \eta + (-1)^{\mid \omega \mid} \omega d \eta
\end{equation*} 
for homogeneous forms $\omega$ and $\eta$ holds true. \\
We define the linear operator $b: \Omega^{n}(\A) \ra \Omega^{n-1}(\A)$ by 
\begin{align*}
 b \left( a^{0}da^{1}\cdots da^{n} \right) =& (-1)^{n-1} \left( a^{0} da^{1} \cdots d a^{n-1} a^{n} - a^{n}a^{0} da^{1} \cdots d a^{n-1} \right) \\
					   =& (-1)^{n-1} \left[ a^{0}da^{1} \cdots da^{n-1}, a^{n} \right] \\
					   =& a^{0}a^{1}da^2 \cdots da^{n} + \\
						&\sum^{n-1}_{j=1} (-1)^{j} a^{0}da^{1} \cdots d \left( a^{j}a^{j+1} \right) \cdots da^{n} \\
						&+  (-1)^{n} a^{n}a^{0}da^{1} \cdots d a^{n-1}.
\end{align*}
A short calculation shows $b^{2} = 0$. Next we introduce the $B$-operator $B: \Omega^{n}(\A) \ra \Omega ^{n+1} (\A)$ given by
\begin{equation} \label{B_operator}
B \left( a^{0}da^{1} \cdots d a^{n} \right) = \sum ^{n}_{i = 0} (-1)^{ni} d a^{n+1-i} \cdots da^{n}da^{0} \cdots da^{n-i}.
\end{equation}
One can check that $ b^2 = B^{2} = 0$ and $Bb + bB = 0$. According to these equations we can form the $(B,b)$-bicomplex of $\A$.
\begin{displaymath}
\begin{xy}
 \xymatrix{ \vdots \ar[d] & \vdots \ar[d] & \vdots \ar[d] & \vdots \ar[d] \\
	\Omega^{3}(\A) \ar[d]_{b} & \Omega^{2}(\A) \ar[l]_{B} \ar[d]_{b} & \Omega^{1}(\A) \ar[l]_{B} \ar[d]_{b} & \Omega^{0}(\A) \ar[l]_{B} \\
	\Omega^{2}(\A) \ar[d]_{b} & \Omega^{1}(\A) \ar[d]_{b} \ar[l]_{B} & \Omega^{0}(\A) \ar[l]_{B}  \\
	\Omega^{1}(\A) \ar[d]_{b} & \Omega^{0}(\A) \ar[l]_{B}  \\
	\Omega^{0}(\A)  \\
}
\end{xy}
\end{displaymath}
The \textit{cyclic homology} of $\A$ is the homology of the total complex of the $(B,b)$-bicomplex of $\A$. \\
The corresponding cohomological version, cyclic cohomology, is quite easy to define. If $V$ is a vector space we denote by $V' = Hom(V, \mathbb{C})$ its dual space. If $f: V \rightarrow W $ is a linear map then it induces a linear map $W' \rightarrow V'$ which will be denoted by $f'$. Applying the dual space functor to the $(B,b)$-complex we have the $(B',b')$-complex of an algebra $\A$ which is again a bicomplex, i. e. $b'^2 =0$, $B'^2=0$, and $B'b'+b'B'=0$. The \textit{cyclic cohomology} of $\A$ is the cohomology of the total complex of the dual $(B',b')$-bicomplex of $\A$. In the sequel we simply write $B$ instead of $B'$ and $b$ instead of $b'$ because it is always clear which kind of map, the ordinary or the dual one, is considered.\\

Next we discuss the noncommutative replacement of pseudo differential calculus and the Atiyah-Singer's index theorem. The main source for this part is \cite{connes_index}. In order to make pseudo differential calculus work for spectral triples we must impose three additional constraints. 

\begin{1_ax}
There is an integer $n$ such that the decreasing sequence $(\l_{k})_{k \in \mathbb{N}}$ of eigenvalues of the compact operator $\mid \D \mid ^{-1} $ satisfies
\begin{equation*}
\lambda_{k} = O \left( k^{-1/n} \right)
\end{equation*} when $k \ra \infty $.
\end{1_ax}

The smallest integer which fulfils this condition is called the \textit{dimension} of the spectral triple. In the case of a $p$-dimensional, closed spin manifold, the dimension of the triple $(C^{\infty}(M), L^{2}(\S), \D)$ coincides with the dimension $p$ of the manifold. Since $\D$ has compact resolvent, its kernel $\ker \D$ is finite dimensional and $\AD ^{-1}$, with $\AD = \sqrt{\D^{2}}$, is well defined as $\AD ^{-1}$ on the orthogonal complement of $\ker \D$ and $0$ on $\ker \D$. Since the kernel finite dimensional, it does not influence the asymptotic behaviour of the Dirac operator or its inverse respectively. \\
Now we consider the derivation $\d$ defined by $\d (T) = \lcom \AD, T \rcom $ for an operator $T$ on $\H$.

\begin{2_ax}
Any element $b$ of the algebra generated by $\pi(\A)$ and $\lcom \D, \pi(\A) \rcom $ is contained in the domain of $\d^k$ for all $k \in \NH$, i. e. $\delta^{k}(b)$ is densely defined and has a bounded extension on $\H$.
\end{2_ax}

In the case of a spin manifold $M$ this condition tells us that we should work with $C^{\infty}(M)$-functions only. Combining the dimension and regularity conditions we obtain that the functions 
\begin{equation*}
\zeta_{b}(z)= \tr \left( b \AD ^{-2z} \right)
\end{equation*}
are well defined and analytic for $\text{Re} \, z >  p/2$ and all $b $ in the algebra $\B$ generated by the elements
\begin{equation*}
\delta^k(\pi(a)), \quad \delta^k(\lcom \D, \pi(a) \rcom ) \quad \text{for } a \in \A \, \text{ and } \, k \geq 0.
\end{equation*}
In order to apply the local index theorem \cite{connes_index} of A. Connes and H. Moscovici we need even more. More precisely, we have to introduce the notion of the \textit{dimension spectrum} of a spectral triple. This is the set $\Sigma \subset \CH$ of singularities of functions
\begin{equation*}
 \zeta_{b}(z)= \tr \left( b \AD ^{-2z} \right) \quad \text{Re} \, z >  p/2 \quad b \in \B \, .
\end{equation*}
We assume the following:

\begin{3_ax}
$\Sigma$ is a discrete subset of $\CH$. Therefore, for any element $b$ of the algebra $\B$ the functions 
\begin{equation*}
\zeta_{b}(z) = \tr(b \mid \mathcal{D} \mid ^{-2z}) 
\end{equation*}
extend holomorphically to $\CH \setminus \Sigma$.
\end{3_ax}

By means of the third condition the functions $ \zeta_{b} $ are meromorphic for any $b\in \B$. In the commutative case of a $p$-dimensional manifold $M$ the dimension spectrum $\Sigma$ consists of all positive integers less or equal $p$ and is simple, i. e. the poles of $\zeta_{b}$ are all simple. All spectral triples we are going to work with have simple dimension spectrum. \\
Let us introduce pseudo differential calculus (see \cite{connes_index} for a full discussion). Let $\text{dom}(\delta)$ denote the domain of the map $\delta$ which consists of all $a \in \B(\H)$ such that $\delta(a)$ is densely defined on $\H$ and can be extended to a bounded operator. We shall define the order of operators by the following filtration: For $ r \in \RH $ set
\begin{equation} \label{pseudo}
 OP^{0} = \bigcap_{n \geq 0} \text{dom} (\delta ^{n} ) , \quad OP^{r} = \mid \D \mid ^{r} OP^{0}.
\end{equation}
Due to regularity we obtain $\B \subset OP^{0} $. Additionally, for every $P $ in $OP^{r}$ the operator $ \AD ^{-r} P $ is densely defined and has a bounded extension. \\
Let $\nabla$ be the derivation $\nabla(T) = \lcom \D^2 , T \rcom $ for an operator $T$ on $\H$. Consider the algebra $ \mathfrak{D}$ generated by
\begin{equation*}
\nabla^k(T), \quad T \in \A \,   \text{ or } \,  \lcom \D , \A \rcom 
\end{equation*} 
for $k \geq 0$. The algebra $\mathfrak{D}$ is the analogue of the algebra of differential operators. By corollary B.1 of \cite{connes_index} we obtain $\nabla^k(T) \in OP^k$ for any $T$ in $\A$ or $\lcom \D, \A \rcom$. \\ 
We introduce the notation
\begin{equation*}
\tau_{k} \left( b \mid \D \mid ^{y} \right)  = \res_{z=0} \, z^{k} \, \tr \left( b \mid \D \mid^{y-2z} \right),
\end{equation*}
for $ b \in \mathfrak{D} $, $ y \in \Sigma$ and an integer $ k \geq 0 $. These residues are independent of the choice of the definition of $\AD ^{-1}$ on the kernel of $\D$ because the kernel is finite dimensional. A change of the definition of $\AD^{-1}$ on the kernel of $\AD$ results in a change by a trace class operator. \\

Now we are ready to state the local index theorem proven in \cite{connes_index} for spectral triples which fulfil the three conditions above. First, let us recall the definition of the index. Let $ u \in \A$ be a unitary. Then, by definition of a spectral triple, the commutator $\lcom \D, \pi(u) \rcom$ is bounded. This forces the compression $P \pi(u) P$ of $\pi(u)$ to be a Fredholm operator (if $F = \sgn \D $ then $P=(F+1)/2$). The index is defined by 
\begin{equation*}
\ind ( P \pi(u) P ) = \text{dim ker} \, P \pi(u)P - \text{dim ker} \, P \pi(u)^\ast P.
\end{equation*}
The index map $u \mapsto \ind(P \pi(u)P)$ defines a homomorphism from the $K_1$-group of the pre-$C^{\ast}$-algebra $\A$ to $\ZH$, see \cite{connes}. By the local index theorem this index can be calculated as a sum of residues $\tau_{k}(b \AD^{y})$, for various $b \in \mathfrak{D}$. Since we are only interested in the three dimensional case we state the local index theorem (Corollary II.1 of \cite{connes_index}) only for this case.

\begin{thm} \label{local_index}
Let $(\A, \H, \D)$ be a three dimensional spectral triple fulfilling the three conditions above and let $u \in \A$ be a unitary. Then 
\begin{equation*}
\ind \, P \pi(u) P = \phi_{1}( u^{*}du) - \phi_{3}(u^{*}dudu^{*}du),
\end{equation*}
where
\begin{align*} 
 \phi_{3}(a^0da^1da^2da^3) = & \quad \dfrac{1}{12} \, \tau_0 \left( \pi(a^0) \lcom \D, \pi(a^1) \rcom \lcom \D, \pi(a^2) \rcom \lcom \D, \pi(a_{3}) \rcom \AD^{-3}  \right) \\
&- \dfrac{1}{6} \, \tau_1 \left( \pi(a^0) \lcom \D, \pi(a^1) \rcom \lcom \D, \pi(a^2) \rcom\lcom \D, \pi(a^3) \rcom \AD^{-3}  \right),
\end{align*}
and
\begin{align*}
 \phi_{1} (a^{0}da^{1})  =& \quad \tau_0 \left( \pi(a^0) \lcom \D, \pi(a^1) \rcom \AD^{-1} \right) \\ 
			&- \dfrac{1}{4} \, \tau_0 \left( \pi(a^0) \nabla \left( \lcom \D, \pi(a^1) \rcom \right) \AD^{-1} \right) \\
			&- \dfrac{1}{2} \, \tau_1 \left( \pi(a^0) \nabla \left( \lcom \D, \pi(a^1)\rcom \right) \AD^{-3} \right) \\
			&+ \dfrac{1}{8} \, \tau_0 \left( \pi(a^0) \nabla^{2} \left( \lcom \D, \pi(a^1) \rcom \right) \AD^{-5} \right) \\
			&+ \dfrac{1}{3} \, \tau_1 \left( \pi(a^0) \nabla^{2}\left( \lcom \D, \pi(a^1) \rcom \right) \AD^{-5} \right) \\
			&+ \dfrac{1}{12}\, \tau_2 \left( \pi(a^0) \nabla^{2} \left( \lcom \D, \pi(a^1) \rcom \right) \AD^{-5} \right)
\end{align*}
for $a^0,a^1,a^2,a^3 \in \A$. In addition, the pair $(\phi_3, \phi_1)$ defines a cyclic cocycle in Connes' $(B,b)$-complex.
\end{thm}

In order to prove the main result of this thesis, the gauge invariance of the Chern-Simons action, which is defined in the following section, we need the trace property of the cochain $ \phi_3 $, i. e. 
\begin{equation*}
\phi_3 (AB) = \phi_3 (BA)
\end{equation*}
for all $ A \in \Omega^{i}(\A)$ and $ B \in \Omega^{3-i}(\A) $, with $i = 0, 1,2, 3$.

\begin{lem} \label{trace_property}
 Let $ ( \A, \H , \D ) $ be a three dimensional spectral triple fulfilling the three conditions above. Then for every operator $A$ and $B$ in $\A $ or $\lcom \D, \A \rcom$, and $k \geq 0$ the equation 
\begin{equation*}
\tau_k \left( A B \AD^{-3}  \right)  = \tau_k \left( B A \AD^{-3}  \right) 
\end{equation*}
holds true.
\end{lem}

\begin{proof}
Recall the filtration \eqref{pseudo}. If $P$ is an operator on $\H$ and $P \in OP^{-3-\epsilon}$ for some $\epsilon > 0$, then $P$ is trace class. Let $ A $ and $B $ be in $ \A$ or $ \lcom \D, \A \rcom $. If one uses the relation $ \AD^{-1} A = A \AD^{-1} - \AD^{-1} \lcom \AD , A \rcom \AD^{-1} $ one obtains
\begin{align*}
 \tau_k \left( A B \AD^{-3} \right) =& \res_{z = 0} \, z^{k} \, \tr \left( A B \AD ^{-3-2z} \right) \\
						 			=& \res_{z = 0} \, z^{k} \, \tr \left( B \AD ^{-3-2z} A \right) \\
						 			=& \res_{z = 0} \, z^{k} \, \tr \left( B \AD ^{-2z} A \AD ^{-3} \right) \\
									 & + \sum^{3}_{i=1} \res_{z = 0} \, z^{k} \, \tr \left( B \AD ^{-2z - i} A \AD ^{-3 + ( i - 1 )} \right) \, .
\end{align*}
The operators in the last row are all trace class for $ \text{Re}(z) > -1/2 $, thus the residues of their traces at $ z = 0 $ vanish. By theorem B.1 of \cite{connes_index} we have
\begin{equation*}
\AD ^{-2z} A \AD ^{2z} - A   \in OP^{-1} \, .
\end{equation*}
This implies the identity
\begin{equation*}
 B \AD ^{-2z} A \AD ^{-3} = B \AD ^{-2z} A \AD ^{2z} \AD ^{-3-2z} \equiv B  A \AD ^{-3 - 2z} \, \text{mod}\, OP^{-4} \, . 
\end{equation*}
Summarizing these observations yields the claimed identity. 
\end{proof}

\section{Noncommutative Chern-Simons Action}

If $( \A, \H, \D ) $ is a spectral triple, then $( M_{N}( \A ), \H \otimes \CH^{N}, \D \otimes I_{N} ) $ is a spectral triple over the matrix algebra $M_N(\A)$ of $\A$, for a positive integer $N \in \NH$, satisfying all requirements imposed to $( \A, \H, \D )$. In the sequel, we shall work with the second one, whose differential forms $\Omega^1(M_{N}(\A)) = M_{N} \left( \Omega^1(\A) \right) $ are matrices with values in $\Omega^1(\A)$. The Dirac operator of the extended triple will be also denoted by $\D$ instead of $\D \otimes I_{N}$. A matrix 1-form $A = \left( a_{ij} db_{ij} \right) _{1 \leq i,j \leq N}$ is called hermitian if $A = A ^{\ast}$ with
\begin{equation*}
\left( \left( a_{ij} db_{ij} \right) _{i,j \leq N } \right) ^{\ast} = \left( \left( a_{ji}db_{ji} \right)^{\ast} \right) _{1 \leq i,j \leq N } \, ,
\end{equation*} 
where the $\ast$-operator on $\Omega(\A)$ is defined as follows: 
\begin{equation*}
 (a^0da^1 \cdots da^n )^{\ast} = (-1)^n d a^{n \ast} \cdots da^{1 \ast} a^{0 \ast} \, 
\end{equation*} 
for $a^0, \ldots, a^n \in \A$.

\begin{defn} \label{CS}
 Let $( \A, \H, \D ) $ be a spectral triple satisfying the conditions of section 2.1, $N \in \NH$, and $A \in M_{N} \left( \Omega^1(\A) \right)$ a hermitian matrix of 1-forms. We define a Chern-Simons action as \[S_{CS}(A) = 6 \pi k \phi_{3}\left( AdA +  \frac{2}{3} A ^{3} \right) - 2 \pi k \phi_{1}(A)\] for an integer $k$, with the cyclic cocycle $(\phi_3,\phi_1)$ of the local index theorem \ref{local_index}.
\end{defn} 

Firstly, we have to study the behaviour of the action under gauge transformations. The transformed action should differ from the initial one only by $ 2 \pi m$, with $m \in \ZH$. Secondly, we show that the noncommutative Chern-Simons action is the classical one \eqref{eq_1.1} if the spectral triple $(C^{\infty}(M), L^{2}(\S), \D)$ comes from a $3$-dimensional, closed spin manifold $M$ with gauge group $SU(N)$. \\
We start by introducing the noncommutative replacement of a gauge mapping. The gauge group of $\left( M_N ( \A ), \H \ot \CH^N, \D \ot I_N \right) $ is the group of unitary elements of $ M_{N}( \A )$. An element $u \in  M_{N}( \A )$ is called unitary if $u^\ast u= uu^\ast =1$. A gauge transformation of a 1-form $A \in M_{N} ( \Omega^1(\A) )$ by a unitary $u \in M_{N}( \A )$ is given as follows:
\begin{equation*}
A \mapsto A^u = uAu^{\ast} + u du^{\ast} \, . 
\end{equation*}
If $A \in M_N(\Omega^1(\A))$ is a hermitian matrix of $1$-forms, then this holds true for $A^u$ too.

\begin{thm} 
 Let $( \A, \H, \D ) $ be a spectral triple satisfying the conditions of the second section, $N \in \NH$, and $A \in M_{N} ( \Omega^1(\A) )$ a hermitian matrix of 1-forms. Then under the gauge transformation of $A$ by a unitary $u \in M_{N}(\A)$, the Chern-Simons action becomes 
\begin{equation*}
 S_{CS}(uAu^{\ast} + u du^{\ast}) = S_{CS}(A) + 2 \pi k \,  \ind(P \pi(u) P), 
\end{equation*}
where $P = (1+F)/2$ and $F = \sgn  \D$ is the sign of $\D$. 
\end{thm}

\begin{proof}
First, let us introduce the curvature $F = dA + A^{2}$ of $A$. We have
\begin{equation*}
AdA + \dfrac{2}{3} A^{3} = AF - \dfrac{1}{3}A^{3}.  
\end{equation*}
Under gauge transformation $A$ becomes $uAu^{\ast} + u d u^{\ast}$ and $F$ transforms into $uFu^{\ast}$. This follows from $0 = d1 = d \left( u u^{\ast} \right) =  \left( d u \right)  u^{\ast} + u du^{\ast} $ and the calculation
\begin{align*}
 & \quad d \left( u A u^{\ast} + u d u^{\ast} \right) + \left( u A u^{\ast} + u d u^{\ast} \right)^2  \\
=& \quad du A u^{\ast} + u dA u^{\ast} - u A du^{\ast} + du du^{\ast} + u A^{2} u^{\ast} + u du^{\ast} \left( u A u^{\ast} \right) + u A u^{\ast} \left( u d u^{\ast} \right) \\
&+ u d u^{\ast} \left( u d u^{\ast} \right) \\
=& \quad u A^2  u^{\ast} + u dA u^{\ast} \\
=& \quad u F u^{\ast}
\end{align*}
Therefore, we have 
\begin{align*}
S_{CS} \left( u A u^{\ast} + u d u^{\ast} \right) = & \quad 6 \pi k \phi_{3}\left( \left( u A u^{\ast} + u d u^{\ast}\right)u F u^{\ast} -  \frac{1}{3} \left( u A u^{\ast} + u d u^{\ast}\right) ^{3} \right)  \\
& - 2 \pi k \phi_{1}(u A u^{\ast} + u d u^{\ast}).
\end{align*}
If one uses the trace property of $\phi_3$, lemma \ref{trace_property}, and $du^{\ast}u = - u^{\ast} du $, we obtain
\begin{equation*}
\phi_{3}\left( \left( u A u^{\ast} + u d u^{\ast}\right)u F u^{\ast}  \right) = \phi_3 \left( AF \right) - \phi_3 \left( duF u^{\ast} \right) 
\end{equation*}
and
\begin{align*}
 &-\dfrac{1}{3}\phi_3 \left( \left( u A u^{\ast} + u d u^{\ast}\right) ^{3} \right) \\
 = &-\dfrac{1}{3} \phi_3 \left( A^3 \right) - \dfrac{1}{3} \phi_3 \left( \left( u du^{\ast} \right) ^{3} \right) - \phi_3 \left( u d u^{\ast} u A^{2} u^{\ast} \right) - \phi_3 \left( u A u^{\ast} \left( u d u^{\ast} \right)^{2} \right) \\
 = & -\dfrac{1}{3} \phi_3 \left( A^3 \right) + \dfrac{1}{3} \phi_3 \left( u du^{\ast}dudu^{\ast} \right) - \phi_3 \left( u d u^{\ast} u A^{2} u^{\ast} \right) + \phi_3 \left( du A du^{\ast} \right) 
\end{align*}
Due to the fact that $(\phi_{3}, \phi_{1})$ is a cyclic cocycle, the term $\phi_{1} \left( u A u^{\ast} \right) $ can be rewritten in a more convenient form:
\begin{align*}
- \phi_{1} \left( u A u^{\ast} \right) + \phi_1 \left( A \right) &= - \phi_{1} \left( \left[ uA, u^{\ast} \right] \right) \\  
								 &= \quad b \phi_1 \left( u A du^{\ast} \right) \\
								 &= - B \phi_3 \left( u A du^{\ast} \right) \\
								 &= - 3 \phi_3 \left( d \left( u A \right) du^{\ast} \right) \, ,
\end{align*}
where the last equality follows by the definition \eqref{B_operator} of the operator $B$ and the trace property of $\phi_3$. Therefore, 
\begin{equation*}
\phi_1 \left( u A u^{\ast} \right) = \phi_{1} \left( A \right) + 3 \phi_{3}  \left( d( u A) d u^{\ast}  \right).
\end{equation*}
Gathering all these terms yields 
\begin{align*}
 & S_{CS} \left( u A u^{\ast} + u d u^{\ast} \right) \\
 = & \quad 6 \pi k \phi_{3}\left(AF - \dfrac{1}{3} A^3 \right) - 6 \pi k \phi_3 \left( d u F u^{\ast} \right) + 2 \pi k \phi_3 \left( u du^{\ast}dudu^{\ast} \right) \\
 & - 6 \pi k \phi_3 \left( u d u^{\ast} u A^{2} u^{\ast} \right) + 6 \pi k \phi_3 \left( du A du^{\ast} \right) \\
 & - 2 \pi k \phi_1 \left(A \right) - 6 \pi k \phi_3 \left( d \left( u A \right) du^{\ast} \right) - 2 \pi k \phi_1 \left( udu^{\ast} \right) \\
= & \quad S_{CS}(A) - 2 \pi k \, \ind \left(P \pi \right( u^\ast \left) P \right)  \\
&+ 6 \pi k  \phi_3 \left( du A du^{\ast} - du F u^{\ast} - u d u^{\ast} u A^{2} u^{\ast} \right) - 6 \pi k  \phi_3 \left( d \left( u A \right) du^{\ast} \right)
\end{align*}
Using again of the trace property \ref{trace_property} of $\phi_3$ and the relation $du^{\ast}u = - u^{\ast} du $ one can simplify the last term
\begin{align*}
\phi_3 \left(du A du^{\ast} - du F u^{\ast} - u d u^{\ast} u A^{2} u^{\ast} \right) &= \phi_3 \left(du A du^{\ast} - du \left( d A + A^{2} \right)  u^{\ast} + d u A^{2} u^{\ast} \right) \\
&= \phi_3 \left(du A du^{\ast} - du dA u^{\ast} \right) \\
&= \phi_3 \left(du A du^{\ast} + dA du^{\ast} u \right) \\
&= \phi_3 \left( d(uA)du^\ast \right) \, .
\end{align*}
The index map $u \mapsto \ind(P\pi(u)P)$ is a homomorphism from the $K_1$-group of $\A$ to $\ZH$. Thus we have $- \ind(P \pi(u^{\ast}) P) = \ind(P \pi(u) P)$, and the proof is done.
\end{proof}

Next we have to prove that for spectral triples $(C^{\infty}(M), L^{2}(\S), \D)$ coming from a $3$-dimensional, closed spin manifold $M$ with spinor bundle $S$, the definition \ref{CS} of a noncommutative Chern-Simons action coincides with the classical one \eqref{eq_1.1}. Essentially, this is given by the following theorem, see \cite{connes_index}, remark II.1.

\begin{thm} \label{commutative}
 Let $( \A, \H, \D ) $ consisting of the Dirac operator $\D$ acting on the Hilbert space of $L^{2}$-spinors over a closed spin manifold $M$, of dimension $p$, and with $\A = C^{\infty}(M)$. Then:
\begin{itemize}
 \item the dimension of the triple $( \A, \H, \D ) $ is simple and contained in the set $\left\lbrace n \in \NH; \, n \leq p \right\rbrace $. This forces $\tau_k = 0$ for $k \geq 1$; 
 \item one has 
 \begin{equation*}
 \tau_0 \left( \pi(a^{0}) \nabla^{k_{1}}\lcom \D, \pi(a^{1}) \rcom \ldots \nabla^{k_{n}}\lcom \D, \pi(a^{n}) \rcom \AD^{-n- 2\mid k \mid} \right) = 0,
 \end{equation*}
for $\mid k \mid \neq 0$; with $k = ( k_1 , \ldots, k_n ) \in \NH$ and $n \leq q$.
 \item for $k_1 = k_2 = \ldots = k_n = 0$, one has 
 \begin{equation*}
 \tau_0 \left( \pi(a^0) \lcom \D, \pi(a^1) \rcom \ldots \lcom \D , \pi(a^n) \rcom  \mid \D \mid^{-p} \right) = \nu_p \int_{M} \hat{A}(R) \wedge a^0 da^1 \wedge \ldots \wedge da^n ,
 \end{equation*}
 where $\nu_p$ is a numerical factor depending only on the dimension $p$ of the manifold $M$ and $\hat{A}(R)$ stands for the $\hat{A}$-form of the Riemannian curvature of $M$. 
\end{itemize}
\end{thm}

The $\hat{A}$-form of the Riemannian curvature of a $p$-dimensional Riemannian manifold $M$ is a polynomial in the Pontrjagin forms $p_{j}(TM)$ of the tangent bundle $TM$ with the Levi-Civita connection. These Pontrjagin forms $p_{j}(TM)$ are in the $4j$-th cohomology group $H^{4j}(M)$ of $M$. For further details see \cite{gilkey}, chapter 2. Since we are only interested in the case $p=3$, all Pontrajagin forms $p_j$ vanish for $j \geq 1$ and we get $\hat{A} = 1$. Hence the formulas for the cyclic cocycle $(\phi_3 , \phi_1 )$ adopt the following form:

\begin{align} \label{commutative_cycle}
&\phi_1 = 0 \\ \nn
&\phi_3 \left( a^0 da^1 da^2 da^3 \right) = \dfrac{\nu_3}{12} \int_{M} a^0 da^1 \wedge da^2 \wedge da^3,
\end{align}
 for $a^i \in C^{\infty}(M)$. The only thing we need to know is the exact value of $\nu_3$. In order to achieve this we make use of the results of A. Carey, J. Phillips and F. Sukochev in \cite{carey_2}. From theorem 5.6 in \cite{carey_2} we obtain the identity
\begin{equation*}
\lim_{s \rightarrow 3^+} \tr \left( \AD^{-s} \right) = 3 \tau_{\omega} \left( \AD^{-3} \right), 
\end{equation*} 
where $\tau_{\omega}$ denotes the so called \textit{Dixmier trace}, a trace defined on the so called Schatten class $\mathcal{L}^{(1,\infty)}(\H)$, which is an ideal in $\B(\H)$. The precise meaning of $\tau_{\omega}$ and $\mathcal{L}^{(1,\infty)}(\H)$ does not bother us because we get rid of them by formula 7.1 \cite{carey_2} which states the equality
\begin{equation*}
\tau_{\omega} \left(\mid \D \mid ^{-3} \right) = \dfrac{1}{3( 2 \pi^3)} \int_{M \times S^2} \sigma_3 \left( \mid \D \mid^3 \right)(x, \xi)^{-1} \, d\nu(x)d\xi. 
\end{equation*}
The integral term on the right hand side of the equality is called the \textit{Wodzicki residue} $\text{Wres}\left( \AD^{-3} \right)$ of the pseudo differential operator $\AD^{-3}$. From example 5.16 in \cite{Bondia} we obtain 
\begin{equation*}
\text{Wres} \left( \AD^{-3} \right) =  8 \pi \,  \text{Vol}(M).
\end{equation*}
Combining all these identities we obtain 
\begin{align*}
\tau_0 \left( \mid \D \mid^{-3} \right) =& \lim_{z \rightarrow 0} z \, \tr \left( \AD ^{-3-2z} \right) \\
=&   \frac{1}{2} \lim_{s \rightarrow 3^+} s \, \tr \left( \AD ^{-s} \right) \\
=& \frac{3}{2} \tau_{\omega} \left( \AD^{-3} \right) \\
=& \dfrac{1}{2 ( 2 \pi^3)} \, \text{Wres} \left( \AD^{-3} \right) \\
=& \dfrac{1}{2 \pi^2} \, \text{Vol}(M)
\end{align*}
which forces 
\begin{equation*}
\nu_3 = (2 \pi^2)^{-1}.
\end{equation*}
An additional problem arises. One deals with the extended triple $( C^\infty(M) \otimes M_N(\CH), L^2(\S) \otimes \CH^N, \D \otimes I_N )$ instead of the commutative triple $( C^\infty(M), L^2(\S), \D )$. This spectral triple is no longer commutative. But the cyclic cocycle of the the extended triple is given by the pair $( \tr \otimes \phi_3, \tr \otimes \phi_1)$ where $(\phi_3, \phi_1)$ denotes the cyclic cocycle of the initial, commutative spectral triple, and \lgans $\tr$\rgans$\,$ denotes the trace on the $N \times N$-matrices $M_N(\CH)$. It is easy to verify that the formulas for the extended triple adopt a form analogous to \eqref{commutative_cycle}:
\begin{align*}
&\tr \otimes \phi_1 = 0 \\
&\tr \otimes \phi_3 \left( a^0 da^1 da^2 da^3 \right) = \dfrac{\nu_3}{12} \int_{M} \tr \left( a^0 da^1 \wedge da^2 \wedge da^3 \right),
\end{align*}
with $a^i \in M_N(\Omega^1(\A))$. 

\begin{cor} \label{cor_3.3}
 Let $\left( \A, \H, \D \right) $ consisting of the Dirac operator $\D$ acting on the Hilbert space of $L^{2}$-spinors over a closed spin manifold $M$ of dimension $3$, and $\A = C^{\infty}(M)$. Then the Chern-Simons action of definition \ref{CS} adopts the form
\begin{equation*}
S_{CS}(A) =  \dfrac{ k }{4 \pi} \int_{M} \tr \left( A \wedge dA + \dfrac{2}{3} A \wedge A \wedge A \right)
\end{equation*}
for a hermitian one form $ A \in M_{N} \left( \Omega^1 \left( C^{\infty}(M) \right)  \right)$. 
\end{cor}

\begin{rem}
 One must take care of the sloppy notation in the the corollary. Formally, the hermitian 1-form $ A \in M_{N} \left( \Omega^1 \left( C^{\infty}(M) \right) \right) $ on the left hand side of the equation is a 1-form defined in the noncommutative framework. But due to the Hochschild-Kostant-Rosenberg theorem (see \cite{voigt}) 1-forms in $\Omega^1 \left( C^{\infty}(M) \right)$ do not coincide with ordinary differential $1$-forms on $M$. Therefore, an additional clue of the identity in the corollary above is the different nature of the $3$-forms on both sides of it. We have an element of $ M_{N} \left( \Omega^3 \left( C^{\infty}(M) \right)  \right) $ as input in the cyclic cocycle $\phi_3$ on the left hand side. Instead, the integrand on the right hand side is an ordinary, matrix valued, differential $3$-form. Hence the use of the same symbol $A$ for the 1-forms on both sides of the equation in the corollary above is not absolutely correct.
\end{rem}

\part{Chern-Simons Theory for the quantum group $SU_q(2)$}

In the second part of the paper we apply the general theory of the first one to the quantum group $SU_q(2)$ and its $3$-dimensional spectral triples constructed in \cite{pal} and \cite{Landi_1}.

\section{The Quantum group $SU_q(2)$}

For the overview of the spectral triple constructed in \cite{pal} we follow \cite{connes_sphere}. Let $q$ be a real number $0 \leq q<1$. We start with the presentation of the algebra of coordinates on the quantum group $SU_q (2)$ in the form 
\begin{equation*}
\a^* \a + \b^* \b = 1 \, , \ \a \a^* + q^2 \b \b^* = 1 \, , \ 
\a \b = q \b \a \, , \ \a \b^* = q \b^* \a \, , \ \b \b^* = 
\b^* \b \, .
\end{equation*}
Let us recall the notations for the standard representation of this algebra. One lets $\H$ be the Hilbert space with orthonormal basis $e_{ij}^{(n)}$ where $n \in \frac{1}{2} \, \NH$ varies among half-integers while $i,j \in \{ -n , -n+1 , \ldots , n \}$. Thus the first elements are,
\begin{equation*}
e_{00}^{(0)} \, , \ e_{ij}^{(1/2)} \, , \ i,j \in \left\{ -  \frac{1}{2} , \frac{1}{2} \right\} , \ldots
\end{equation*}
The following formulas define a unitary representation on $\H$,
\begin{equation*}
\a e_{ij}^{(n)} = a_+ (n,i,j) \, e_{i-\frac{1}{2} , j - \frac{1}{2}}^{\left( n + \frac{1}{2} \right)} + a_- (n,i,j) 
\, e_{i-\frac{1}{2} , j - \frac{1}{2}}^{\left( n - \frac{1}{2} \right)}
\end{equation*}
\begin{equation*}
\b e_{ij}^{(n)} = b_+ (n,i,j) \, e_{i+\frac{1}{2} , j - \frac{1}{2}}^{\left( n + \frac{1}{2} \right)} + b_- (n,i,j) 
\, e_{i+\frac{1}{2} , j - \frac{1}{2}}^{\left( n - \frac{1}{2} \right)}
\end{equation*}
where the explicit form of $a_{\pm}$ and $b_{\pm}$ is
\begin{equation*}
a_+ (n,i,j) = q^{2n+i+j+1} \, \frac{(1-q^{2n-2j+2})^{1/2} (1-q^{2n-2i+2})^{1/2}}{ (1-q^{4n+2})^{1/2} (1-q^{4n+4})^{1/2}}
\end{equation*}
\begin{equation} \label{eq_3.1}
a_- (n,i,j) = \frac{(1-q^{2n+2j})^{1/2} (1-q^{2n+2i})^{1/2}}{(1-q^{4n})^{1/2} (1-q^{4n+2})^{1/2}}
\end{equation}
and
\begin{equation*}
b_+ (n,i,j) = -q^{n+j} \, \frac{(1-q^{2n-2j+2})^{1/2} (1-q^{2n+2i+2})^{1/2}}{(1-q^{4n+2})^{1/2} (1-q^{4n+4})^{1/2}}
\end{equation*}
\begin{equation} \label{eq_3.2}
b_- (n,i,j) = q^{n+i} \, \frac{(1-q^{2n+2j})^{1/2} (1-q^{2n-2i})^{1/2}}{(1-q^{4n})^{1/2} (1-q^{4n+2})^{1/2}} \, .
\end{equation}
Note that $a_-$ does vanish if $i=-n$ or $j=-n$, which gives meaning to $a_- (n,i,j) \, e_{i-\frac{1}{2} , j-\frac{1}{2}}^{\left( n-\frac{1}{2}\right)}$ for these values while $i-\frac{1}{2} \notin \left[ - \left( n-\frac{1}{2} \right) , n - \frac{1}{2} \right]$ or $j - \frac{1}{2} \notin \left[ - \left( n-\frac{1}{2} \right) , n - \frac{1}{2} \right]$. Similarly $b_-$ vanishes for $j=-n$ or $i=n$.\\
Let now, as in \cite{pal}, $\D$ be the diagonal operator on $\H$ given by,
\begin{equation*}
\D (e_{ij}^{(n)}) = (2 \, \d_0 (n-i)-1) \, 2n \,\, e_{ij}^{(n)}
\end{equation*}
where $\d_0 (k) = 0$ if $k \ne 0$ and $\d_0 (0) = 1$. It follows from \cite{pal} that the triple $(C^\infty(SU_q(2)) , \H , \D)$ is a spectral triple, where $C^\infty(SU_q(2))$ is the pre-$C^\ast$-algebra generated by the elements $\a$ and $\b$. Due to Theorem 3.2 and Theorem 7.1 in \cite{connes_sphere} we obtain:

\begin{thm}
The spectral dimension of the triple $( C^\infty(SU_q(2)), \H, \D)$ is simple and equal to $\Sigma = \set{1,2,3}$.
\end{thm}

In the sequel we have to treat the case $q = 0$ separately from the general case, $ 0 < q < 1$.

\medskip

\subsection{Local index formula for $SU_0(2)$}

For $q = 0$ the representations of $\a$ and $\b$ adopt a comparatively simple form.
\begin{eqnarray*}
&&a_+ (n,i,j) = 0 \\
&&a_- (n,i,j) = 0 \ {\rm if} \ i=-n \ {\rm or} \ j = -n \\
&&a_- (n,i,j) = 1 \ {\rm if} \ i \ne -n \ {\rm and} \ j \ne -n 
\end{eqnarray*}
\begin{eqnarray} \label{eq_3.3}
&&b_+ (n,i,j) = 0 \ {\rm if} \ j \ne -n  \\
&&b_+ (n,i,j) = -1 \ {\rm if} \ j = -n \nn\\
&&b_- (n,i,j) = 0 \ {\rm if} \ i\ne -n \ {\rm or} \ j = -n  \nn \\
&&b_- (n,i,j) = 1 \ {\rm if} \ i=-n , j \ne -n  \nn \, . 
\end{eqnarray}
Thus for $q=0$ the operators $\a$ and $\b$ on $\H$ are given by
\begin{equation*}
\a e_{ij}^{(n)} = e_{i - \frac{1}{2} , j - \frac{1}{2}}^{\left( n - \frac{1}{2} \right)} \qquad \hbox{if} \quad i > -n , j > -n
\end{equation*}
and $\a e_{ij}^{(n)} = 0$ if $i = -n$ or $j=-n$.
\begin{equation*}
\b e_{ij}^{(n)} = 0 \qquad \hbox{if} \quad i \ne -n \quad \hbox{and} \quad j \ne -n
\end{equation*}
\begin{equation*}
\b e_{-n,j}^{(n)} = e_{-\left( n - \frac{1}{2} \right)  , j - \frac{1}{2}}^{\left( n - \frac{1}{2} \right)} \qquad \hbox{if} \quad j 
\ne -n
\end{equation*}
and
\begin{equation*}
\b e_{i,-n}^{(n)} = - e_{i+ \frac{1}{2} , - \left( n+ \frac{1}{2} 
\right)}^{\left( n + \frac{1}{2} \right)} \, .
\end{equation*}
By construction $\b \b^* = \b^* \b$ is the projection $e$ on the subset $\{ i = -n$ or $j = -n\}$ of the basis. Moreover, $\a$ is a partial isometry with initial support $1-e$ and final support $1 = \a \a^*$. The basic relations between $\a$ and $\b$ are,
\begin{equation*}
\a^* \a + \b^* \b = 1 \, , \ \a \a^* = 1 \, , \ \a \b = \a \b^* = 0 \, , 
\ \b \b^* = \b^* \b \, .
\end{equation*}
For $f \in C^{\infty} (S^1)$, $f = \sum \wh f_n \, e^{in \t}$, we let
\begin{equation*}
f(\b) = \sum_{n > 0} \wh f_n \, \b^n + \sum_{n < 0} \wh f_n \, \b^{*(-n)} + \wh f_0 \, e \, ,
\end{equation*}
and the map $f \ra f(\b)$ gives a (degenerate) representation of $C^{\infty} (S^1)$ on $\H$. Let $C^\infty(SU_0(2))$ denote the linear span of the elements
\begin{equation*}
a = \sum_{k,\ell \geq 0} \a^{*k} f_{k\ell} (\b) \, \a^{\ell} + \sum_{\ell \geq 0} \l_{\ell} \, \a^{\ell} + \sum_{k > 0} \l'_k \, 
\a^{*k}
\end{equation*}
where $\l$ and $\l'$ are sequences (of complex numbers) of rapid decay and $(f_{k\ell})$ is a sequence of rapid decay with values in $C^{\infty} (S^1)$. Due to proposition 3.1 of \cite{connes_sphere} $C^\infty(SU_0(2))$ is a $\ast$-subalgebra of $\B(\H)$ closed under holomorphic calculus. We define a linear map $\sg$ from $C^\infty(SU_0(2))$ to $C^\infty(S^1)$ by
\begin{equation} \label{eq_3.4}
\sg (a) = \sum_{\ell \geq 0} \l_{\ell} \, u^{\ell} + \sum_{k > 0} \l'_k \, u^{-k} \, ,
\end{equation}
where $u = e^{i\t}$ denotes the unitary generator of $C^\infty(S^1)$. In the proof of proposition 3.1 in \cite{connes_sphere} it is actually shown that $\sg$ is a $\ast$-homomorphism, which becomes quite important for the description of the cyclic cocycle $( \phi_3, \phi_1 )$ for $SU_0(2)$. We start this description with the following formulas defining a cyclic cocycle $\tau_1$ on $C^\infty(SU_0(2))$,
\begin{equation*}
\tau_1 (\a^{*k}, x)=\tau_1 (x,\a^{*k})= \tau_1 (\a^{l}, x)=\tau_1 (x,\a^{l})=0,
\end{equation*}
for all integers $k$, $l$ and any $x \in C^\infty(SU_0(2))$,
\begin{equation*}
\tau_1 (\a^{*k} f(\b) \a^{\ell} , \a^{*k'}  g(\b) \a^{\ell'}) = 0
\end{equation*}
unless  $\ell'=k$, $k'=\ell$ and
\begin{equation*}
\tau_1 (\a^{*k} f(\b)\a^{\ell} , \a^{*\ell} g(\b) \a^k) = \frac{1}{\pi i} \int_{S^1} f \,\, {\rm d} g \,\, .
\end{equation*}
Let $\vp_0$ be the 0-cochain given by $\vp_0 (\a^{*k} f(\b) \a^{\ell}) = 0 \,\, \hbox{unless} \,\, k = \ell $ and,
\begin{equation*}
\vp_0 (\a^{*k} f(\b) \a^k) =  \rho (k)\, \, \frac{1}{2\pi}  \int_{0}^{2 \pi} f \, \,{\rm d} \t ,
\end{equation*}
where $\rho (j) =\frac{2}{3} - j - j^2$. Finally, let $\vp_2$ be the 2-cochain given by the pull back by $\sg$ of the 
cochain  
\begin{equation*} 
\rho(f_0,f_1,f_2) = \frac{-1}{24}\frac{1}{2\pi i} \, \int_0^{2 \pi} f_0 f'_1 f''_2 \, {\rm d} \t
\end{equation*}
on  $C^{\infty} (S^1)$. Using these definitions, theorem 4.1 of \cite{connes_sphere} provides us with following result.

\begin{thm} \label{thm_3.2}
The local index formula \ref{local_index} of the spectral triple $(C^\infty(SU_0(2)) , \H , \D)$ is given by the cyclic cocycle $\tau_1$ up to the coboundary of the cochain $(\vp_0,\vp_2)$. The precise equations are
\begin{equation*}
\vp_1= \,\tau_1 + b \vp_0 + B \vp_2 \,, \qquad \vp_3= \, b \vp_2 \,.
\end{equation*}
\end{thm}

\medskip

\subsection{Local index formula for $SU_q(2), \, 0 < q <1 $}

Let $\Br_q$ be the algebra generated by the elements $\d^k(a)$ for $a \in C^\infty(SU_q(2))$, with the derivation $\d$ given by the commutator bracket $\d(\cdot) = \lcom \AD, \cdot \rcom$. By construction, the generators $\a$ and $\b$ of $C^\infty(SU_q(2))$ are of the form, 
\begin{equation*}
\a = \a_+ + \a_- \, , \qquad \b = \b_+ + \b_-
\end{equation*}
where,
\begin{equation*} 
\d (\a_{\pm}) = \pm \a_{\pm} \, , \qquad \d (\b_{\pm}) = \pm \b_{\pm} \, .
\end{equation*}
The explicit form of $\a_{\pm}$, $\b_{\pm}$ is, using $\frac{n}{2}$ instead of $n$ for the notation of the half-integer,
\begin{equation*}
\a_{\pm} (e_{(i,j)}^{(n/2)}) = a_{\pm} \left( n/2 , i , j \right) \, e_{\left( i - \frac{1}{2} , j - \frac{1}{2} \right)}^{\left( \frac{n \pm 1}{2} \right)}
\end{equation*}
\begin{equation*}
\b_{\pm} (e_{(i,j)}^{(n/2)}) = b_{\pm} \left( n/2 , i , j \right) \, e_{\left( i + \frac{1}{2} , j - \frac{1}{2} \right)}^{\left( \frac{n \pm 1}{2} \right)}
\end{equation*}
where $a_{\pm}$, $b_{\pm}$ are as in \eqref{eq_3.1} and \eqref{eq_3.2} above. Thus the algebra $\Br_q$ is generated by the operators $\a_{\pm}$, $\b_{\pm}$ and their adjoints. Since we want to compute \textit{local formulas} of operators in $\Br_q$ we are entitled to mod out smoothing operators
\begin{equation*}
OP^{-\infty} = \bigcap_{k > 0} OP^{-k}
\end{equation*}
where the operator spaces $OP^{-k}$ are given by equation \eqref{pseudo} in the discussion of pseudo differential calculus in the second section. This simplification leads us to the cosphere bundle $C^\infty(S^\ast_q)$.\\
We introduce the following representations $\pi_{\pm}$ of $C^\infty(SU_q(2))$ . In both cases the Hilbert spaces are $\H_{\pm} = \ell^2 (\NH)$ with basis $(\e_x)_{x \in \NH}$ and the representations are given by,
\begin{equation} \label{pi_pl}
\pi_{\pm} (\a) \, \e_x = (1-q^{2x})^{1/2} \, \e_{x-1} \qquad \text{for all } x \in \NH
\end{equation}
\begin{equation*}
\pi_{\pm} (\b) \, \e_x = \pm \,\,q^x \, \e_x \qquad  \text{for all }x \in \NH \, .
\end{equation*}
Additionally, one defines a $\ZH$-grading on $\Br_q$ by means of the one-parameter group of automorphisms $\gamma_t$ given by
\begin{equation*} 
\gamma_t(P) = e^{i t \AD} P e^{-it \AD} \quad \text{for } P \in \Br_q
\end{equation*}
For the corresponding $\ZH$-grading one has
\begin{equation} \label{grading}
\deg(\a_{\pm}) = \pm 1  \quad \deg(\b_{\pm}) = \pm 1 
\end{equation}
From the definition of the representations $\pi_\pm$ one obtains $b - b^\ast \in \ker \pi_\pm$. Hence, the representations $\pi_\pm$ are not faithful. Let denote by 
\begin{equation*}
C^\infty(D^2_{q\pm}) = C^\infty(SU_q(2)) / \ker(\pi_{\pm})
\end{equation*} 
the quotient algebras and $r_\pm$ the restriction morphisms. We are prepared to restate proposition 6.1 of \cite{connes_sphere}

\begin{prop} \label{prop_3.3}
The following formulas define an algebra homomorphism $\rho$ from $\Br_q$ to
\begin{equation*}
C^{\infty} (D_{q+}^2) \otimes C^{\infty} (D_{q-}^2) \otimes C^{\infty} (S^1) \, ,
\end{equation*}
\begin{equation*}
\rho (\a_+) = -q \b^* \otimes \b \otimes u \, , \quad \rho (\a_-) = \a \otimes \a \otimes u^* \, ,
\end{equation*}
\begin{equation*}
\rho (\b_+) = \a^* \otimes \b \otimes u \, , \quad \rho (\b_-) = \b \otimes \a \otimes u^* \, ,
\end{equation*}
where we omitted $r_+ \otimes r_-$.
\end{prop}

\begin{defn}
Let $C^{\infty} (S_{q}^*)$ be the range of $\rho$ in $C^{\infty} (D_{q+}^2) \ot C^{\infty} (D_{q-}^2) \ot C^{\infty} (S^1)$.
\end{defn}

There is a symbol map $\sg:C^\infty(D^2_{q\pm}) \ra C^\infty(S^1) $ which is given as a $\ast$-homomorphism on the generators of $C^\infty(D^2_{q\pm})$ by 
\begin{equation} \label{eq_3.5}
\sg(r_\pm(\a)) = u \, , \qquad \sg(r_\pm(\b)) = 0
\end{equation}
where $u$ is the unitary generator of $C^\infty(S^1)$. The representations $\pi_\pm$ in $\ell^2(\NH)$, with basis $\set{\e_x: x \in \NH}$, induce representations on $C^\infty(D^2_{q\pm})$ which we also denote with $\pi_\pm$. We define two linear functionals $\tau_1$ and $\tau_0$ on $C^{\infty}(D_{q \pm}^2)$ by 
\begin{equation} \label{tau_1}
\tau_1 (a) = \frac{1}{2 \pi} \int_0^{2\pi} \sg (a) \, d \t \qquad \fl a \in C^{\infty} (D_q^2)
\end{equation}
and
\begin{equation} \label{tau_0}
\tau_0 (a) = \lim_{N \ra \infty} \tr_N (\pi(a)) - \tau_1 (a) \, N \, ,
\end{equation}
where
\begin{equation*}
\tr_N (a) = \sum_0^N \langle a \, \e_x , \e_x \rangle \, .
\end{equation*}
$($where we omitted $\pm$ in the above formulas$)$.

\begin{thm} \label{integrals_su_q_2}
Let $b \in \Br_q$, $\rho (b) \in C^{\infty} (S_q^*)$ its symbol. Then let $\rho (b)^0$ be the component of degree $0$, with the grading on $C^{\infty} (S_q^*)$ induced by the one on $\Br_q$. Then one has
\[
\ncint  b \, \AD^{-3} = (\tau_1 \otimes \tau_1)(r\rho (b)^0)
 \]
\[
\ncint b \, \AD^{-2} = (\tau_1 \otimes \tau_0 + \tau_0 \otimes \tau_1)(r\rho (b)^0)
 \]
\[
\ncint b \, \AD^{-1} = (\tau_0 \otimes \tau_0)(r\rho (b)^0)\, ,
 \]
 with the natural restriction homomorphism $r: C^{\infty} (S_q^*) \ra C^{\infty} (D_{q+}^2) \otimes C^{\infty} (D_{q-}^2)$ and the Wodzicki residue
\begin{equation*}
\ncint P = \res_{z=0} \tr \left(P \AD^{-z} \right)
\end{equation*}
\end{thm}

From theorem \ref{integrals_su_q_2} one gets the explicit form of the local index formula for $SU_q(2)$. The only difficulty is the definition of a cycle $(\Omega, d, \int)$ (see \cite{connes}, chapter III for a careful discussion). More precisely, let us define the $C^\infty(SU_q(2))$-bimodule $\Omega^1$ whose underlying linear space is given by the direct sum $\Omega^1 = C^\infty(SU_q(2)) \oplus \Omega^{(2)} (S^1)$ where $\Omega^{(2)} (S^1)$ is the space of differential forms $f(\t) \, {\rm d} \t^2$ of weight 2 on $S^1$. The bimodule structure is defined by
\begin{eqnarray} \label{eq_3.6}
a (\xi , f) = (a\xi , \sg (a) f) \nn \\ 
  (\xi,f) b = (\xi b , -i \, \sg (\xi) \, \sg (b)' + f \sg (b))
\end{eqnarray}
for $a,b \in C^\infty(SU_q(2))$, $\xi \in C^\infty(SU_q(2))$ and $f \in \Omega^{(2)} (S^1)$ and the map $\sg$ defined in \eqref{eq_3.5} omitting $r_\pm$ in the formula. The differential $d$ is then given by
\begin{equation} \label{eq_3.7}
da = \partial a + \frac{1}{2} \, \sg (a)'' \, {\rm d} \t^2
\end{equation}
as in a Taylor expansion. The functional $\int: \Omega^1 \ra \CH$ is defined by
\begin{equation} \label{eq_3.8}
\int (\xi , f) = \tau (\xi) + \frac{1}{2\pi i} \int f \, {\rm d} \t \, ,
\end{equation}
with the linear map $\tau(a) = \tau_0( r_-(a^0)) $, for $a \in C^\infty(SU_q(2))$, where $a^0$ denotes the component of $a$ of degree 0, with respect to the grading on $C^\infty(SU_q(2))$ induced by the derivative $\partial = \partial_\b - \partial_\a$. From proposition 8.1 in \cite{connes_sphere}, we obtain:

\begin{prop}
The triple $(\Omega, d, \int)$ is a cycle, i. e. $\Omega = C^\infty(SU_q(2)) \oplus \Omega^1$ equipped with $d$ as a graded differential algebra (with $\Omega^0 = C^\infty(SU_q(2))$) and the functional $\int$ is a closed graded trace on $\Omega$.
\end{prop}

We let $\chi$ be the cyclic $1$-cocycle which is the character of the above cycle, explicitly
\begin{equation*}
\chi(a^0da^1) = \int a^0da^1 \qquad \fl a^0,a^1 \in C^\infty(SU_q(2)) \, .
\end{equation*} 
Similar to the case $q=0$ we define the cochains
\begin{eqnarray*}
&&\vp_0 (a) =  \tr(\,a\, \vert D \vert^{-s})_{s=0},\\
&& \vp_2(a^0da^1da^2)=\,- \frac{1}{24} \ncint a^0  \,\d(a^1) \,\d^2( a^2) \,  \, \AD^{-3} \,.
\end{eqnarray*}

Now, we obtain a form of the local index formula similar to the case $q = 0$ (Theorem 8.2 in \cite{connes_sphere})

\begin{thm} \label{local_index_q}
The local index formula of the spectral triple $(C^\infty(SU_q(2)) , \H , \D)$ is given by the cyclic cocycle $\chi$ up to the coboundary of the cochain $(\vp_0,\vp_2)$.
\end{thm}

\section{Chern-Simons action for $SU_q(2)$}

We start this section with the construction of symbol maps $\sg_q: \Br_q \ra C^\infty(S^1)$, for $0 \leq q < 1$ and the $\ast$-sub-algebras $\Br_q$ (as described in subsection 4.2) of $\B(\H)$ generated by $\d^k(a)$, with $a \in C^\infty(SU_q(2))$ and the derivation $\d(\cdot) = \lcom \AD, \cdot \rcom$. \\

First let $q = 0$. From \eqref{eq_3.4} we have a symbol map $\sg: C^\infty(SU_0(2)) \ra C^\infty(S^1)$. We extend $\sg$ to a $\ast$-homomorphism $\sg_0: \Br_0 \ra C^\infty(S^1)$ on the generators of $\Br_0$ by
\begin{equation*}
\sg_0(\d^m(a)) = i^m \sg_0(a)^{(m)} \, ,
\end{equation*}
where $f^{(m)}$ denotes the $m$-th derivative of $f \in C^\infty(S^1)$.\\
In the general case, $0 < q < 1$, we establish a definition of $\sg_q$ as follows. From Proposition \ref{prop_3.3}, equation \eqref{eq_3.5} and the restriction homomorphism $r$ in theorem \ref{integrals_su_q_2} we can define the following sequence of maps
\begin{equation*} 
\Br_q \build \longrightarrow_{}^\rho C^{\infty} (S_q^*) \build \longrightarrow_{}^{r} C^{\infty} (D_{q+}^2) \otimes C^{\infty} (D_{q-}^2) \build \longrightarrow_{}^{\sg \otimes \sg} C^\infty( S^1) \ot C^\infty(S^1) \cong C^\infty(S^1 \times S^1) \, .
\end{equation*}
The composition-map $\tsg = (\sg \ot \sg) \circ r \circ \rho : \Br_q \ra C^\infty(S^1 \times S^1)$ acts on the generators $\a_\pm$ and $\b_\pm$ of $\Br_q$ as follows
\begin{equation*}
\tsg (\a_+) = \sg \ot \sg (-q \b^* \otimes \b ) = 0\, , \quad \tsg (\a_-) = \sg \ot \sg ( \a \ot \a ) = e^{i \nu} e^{i \p } \, ,
\end{equation*}
\begin{equation*}
\tsg (\b_+) = \sg \ot \sg (\a^* \ot \b ) = 0 \, , \quad \tsg (\b_-) = \sg \ot \sg (\b \otimes \a ) = 0 \, .
\end{equation*}
The unitary $e^{i ( \nu  + \p)}$ is the generator of the image of $\tsg$ in $C^\infty(S^1 \times S^1)$. Hence, via the $*$-homomorphism which maps $e^{i ( \nu + \p )} \mapsto e^{i \t}$, we can identify the image of $\tsg $ with $C^\infty(S^1)$. Let us denote by $\sg_q: \Br_q  \ra C^\infty(S^1)$ the map we obtain from $\tsg$ by means of this identification.

\begin{lem} \label{symbol}
Let $(C^\infty(SU_q(2)), \H, \D)$ be the spectral triple of $SU_q(2)$, with $0 \leq q < 1$, and the algebra $\Br_q$ generated by $\d^k(a)$, with $a \in C^\infty(SU_q(2))$ and the derivation $\d(\cdot) = \lcom \AD, \cdot \rcom$. $\Br_q$ is generated by $\a_\pm$ and $\b_\pm$ with 
\begin{equation*}
\a_{\pm} (e_{(i,j)}^{(n/2)}) = a_{\pm} \left( n/2 , i , j \right) \, e_{\left( i - \frac{1}{2} , j - \frac{1}{2} \right)}^{\left( \frac{n \pm 1}{2} \right)}
\end{equation*}
\begin{equation*}
\b_{\pm} (e_{(i,j)}^{(n/2)}) = b_{\pm} \left( n/2 , i , j \right) \, e_{\left( i + \frac{1}{2} , j - \frac{1}{2} \right)}^{\left( \frac{n \pm 1}{2} \right)}
\end{equation*}
where $a_{\pm}$, $b_{\pm}$ are as in \eqref{eq_3.1} and \eqref{eq_3.2}. The $\ast$-homomorphisms $\sg_q: \Br_q \ra C^\infty(S^1)$, for $0 \leq q < 1$, act on these generators of $\Br_q$ by
\begin{align*}
\sg_q: \Br_q  &\ra C^\infty(S^1) \\
		\a_- &\mapsto e^{i \t}   \\
		\a_+,b_\pm &\mapsto 0  
\end{align*} 
and fulfil the following properties:
\begin{align}
 \ncint \, b \AD^{-3} =& \dfrac{1}{2\pi} \int \sg_q(b) \, d\t \nn \\
 \sg_q(\d^m(b)) =& i^m \sg_q(b)^{(m)} \qquad \fl b \in \Br_q \, . \label{prop_symbol}
\end{align}
with the Wodzicki residue $\ncint$ given in theorem \ref{integrals_su_q_2}.
\end{lem}

\begin{proof}
\textbf{The case $q = 0$}. From \eqref{eq_3.3} we obtain the identities $\a = \a_-$ and $\a_+ = 0$. In addition, from \eqref{eq_3.3} it is easy to check that $\d(\a) = - \a$ and $\d(\b_\pm) = \pm \b_\pm $. This proves the second equality of \eqref{prop_symbol}.\\
The proof of theorem 3.2 in \cite{connes_sphere}, especially the equations (3.46) and (3.47) there, provides us with the result that any operator $B$ in $\ker \sg_0$ can be written as
\begin{equation*}
B = \sum_{k,l \geq 0} \a^{\ast k} b_{kl} \a^l 
\end{equation*} 
with a rapidly decreasing sequence of bounded operators $(b_{kl})$ whose support is contained in $e \H$, with the projection $e = \b \b^\ast = \b^\ast \b$. Since $e$ commutes with $\AD^{-3}$ we have $e \AD^{-3} = e \AD^{-3} e$. Hence, the operator $e \AD^{-3}$ is positive. In addition $e \AD^{-3}$  is trace class, because
\begin{align*}
\tr \left( e \AD^{-3} \right) =& \sum_{n \in \frac{1}{2}\NH } \sum_{i=-n}^n \sum_{j=-n}^n \langle e^{(n)}_{i,j}, e \AD^{-3} e^{(n)}_{i,j} \rangle \\
=& \sum_{n \in \frac{1}{2}\NH} \sum_{i=-n}^n \sum_{j=-n}^n  \dfrac{1}{(2n)^{3}} \langle e^{(n)}_{i,j}, e \, e^{(n)}_{i,j} \rangle \\
=& \sum_{n \in \frac{1}{2}\NH} \sum_{i=-n+1}^n \sum_{j=-n}^n \dfrac{1}{(2n)^{3}}\left(\langle e^{(n)}_{-n,j},e^{(n)}_{-n,j} \rangle +  \langle e^{(n)}_{i,-n} , e^{(n)}_{i,-n} \rangle \right) \\
=& \sum_{n \in \frac{1}{2}\NH } \dfrac{4n+1}{(2n)^{3}} \\
=& 2 \,\sum_{m \in \NH } \dfrac{1}{m^{2}} + \sum_{m \in \NH } \dfrac{1}{m^{3}} < \infty.
\end{align*} 
The space of trace class operators is an ideal in $\B(\H)$. Hence, 
\begin{align*}
\a^{\ast k} b_{kl} \a^l \AD^{-3}=& \a^{\ast k} b_{kl}e \AD^{-3} \AD^3 \a^l  \AD^{-3} \\
=& \a^{\ast k} b_{kl}e \AD^{-3}\underbrace{\left( \sum_{k=0}^3 \d^k(\a^l) \AD ^{-k} \right)}_{\in \B(\H)}
\end{align*}
is trace class for all $ k,l \geq 0 $ and therefore the operator $B \AD^{-3}$ as well. \\
For $b \in \Br_0$ and $f := \sg_0(b) \in C^\infty(S^1)$ we introduce the notation 
\begin{equation*}
f(\a) = \sum_{\ell \geq 0} \hat{f}_{\ell} \, \a^{\ell} + \sum_{k > 0} \hat{f}_{-k} \, \a^{*k}\, ,
\end{equation*}
where $(\hat{f}_k)_k$ denotes the coefficients of the Fourier decomposition of $f$. We have $b - f(\a) \in \ker \sg_0$, thus the operator $(b - f(\a)) \AD^{-3} $ is trace class. Additionally, from the definition of the representation of $\a$ in \eqref{eq_3.3} and equation (4.76) in \cite{connes_sphere} one obtains
\begin{equation*}
\tr \left( \a^k \AD^{-3} \right) = \tr \left( \a^{\ast k} \AD^{-3} \right) = \d_k \qquad \fl k \geq 0  
\end{equation*}
with the Kronecker symbol $\d_k$ and $\a^0 = \a^{\ast 0} = 1$. Summarizing these considerations yields:
\begin{equation*}
\ncint b \AD^{-3} = \ncint f(\a) \AD^{-3} =  \hat{f}_0 = \dfrac{1}{2\pi} \int_0^{2\pi} f(\t) \, d\t \, .
\end{equation*} 
This proves the first equation of \eqref{prop_symbol} and therefore the lemma for the case $q = 0$.\\

\textbf{The case $0 < q <1$}: From the definitions \eqref{eq_3.1} and \eqref{eq_3.2} of $\a_\pm$ and $\b_\pm$ one obtains $\d(a_\pm) = \pm \a_\pm $ and $\d(b_\pm) = \pm b_\pm$. Hence the second identity of \eqref{prop_symbol} follows by the definition of $\sg_q$.\\
Let $b $ be a word in the variables $\a_\pm,\a_\pm^\ast, \b_\pm$ and $\b_\pm^\ast$, i. e. 
\begin{equation*}
b = \prod_{i = 1}^N x_i^{n_i} \, ,
\end{equation*}
with $x_i \in \set{\a_\pm,\a_\pm^\ast, \b_\pm,\b_\pm^\ast}$, $n_i \geq 0$ and the convention $x_i^0 = 1$. We easily obtain from theorem \ref{integrals_su_q_2} and the definition of the symbol map $\sg_q$
\begin{equation*}
\ncint b \AD^{-3} = 0 = \dfrac{1}{2\pi} \int_0^{2\pi} \underbrace{\sg_q(b)}_{=0} \, d\t
\end{equation*}
if there is at least one index $i_0 \in \set{1, \ldots, N}$ such that $n_{i_0} > 0$ and $x_{i_0} \in \set{\a_+,\a_+^\ast,\b_\pm, \b_\pm^\ast}$. Let us consider the case when all $x_i$ contained in the set $\set{\a_-,\a_-^\ast}$. Due to \eqref{grading} the degree $\deg(b)$ of $b$ is given by the sum
\begin{equation*}
\sum_{i=1}^N \deg(x_i) n_i \,.
\end{equation*}
Applying again theorem \ref{integrals_su_q_2} and the definition of $\sg_q$ we compute:
\begin{align*}
\ncint b \AD^{-3} =& \tau_1 \ot \tau_1 (r\rho(b)^0) \\
=& \d_{\deg(b)} \tau_1 \ot \tau_1 (r\rho(b)) \\
=& \d_{\deg(b)} \dfrac{1}{2\pi} \int_0^{2\pi} e^{ \deg(b) i \nu}\, d\nu \, \dfrac{1}{2\pi} \int_0^{2\pi} e^{ \deg(b) i \phi} \, d\phi \\
=& \dfrac{1}{4 \pi^2} \int_0^{2 \pi} e^{\deg(b) i (\nu + \phi) } \, d\nu d\phi \\
=& \dfrac{1}{2 \pi} \int^{2\pi}_0  e^{\deg(b) i \t } \, d\t \\
=& \dfrac{1}{2 \pi} \int^{2\pi}_0  \sg_q(b) \, d\t \,,
\end{align*}
with the Kronecker-delta symbol $\d_k$. Since the space $\Br_q$ is the linear span of elements like $b$, the second equality of \eqref{prop_symbol} is proved.
\end{proof}

If we apply the lemma to compute the cochain $\vp_2$ of theorem \ref{local_index_q}, we obtain for $ 0 < q < 1$
\begin{equation*}
\vp_2 (a^0da^1da^2) = - \dfrac{1}{24} \ncint a^0\d(a^1)\d^2(a^2) \AD^{-3} = - \dfrac{1}{24} \dfrac{1}{2 \pi i} \int_0^{2\pi} \sg_q(a^0)\sg_q(a^1)'\sg_q(a^2)'' \, d \t \, ,
\end{equation*}
i. e. the same formula as in the case $q = 0$.

\begin{cor} \label{cor_4.2}
Let $(C^\infty(SU_q(2)), \H, \D)$ the spectral triple for the quantum group $SU_q(2)$ and $\sg_q: \Br_q \ra C^{\infty}(S^1)$ the symbol map from lemma \ref{symbol} for $0 \leq q < 1$. The cochain $\vp_2$ of theorem \ref{thm_3.2} and \ref{local_index_q} is given by the pullback by $\sg_q$ of the cochain
\begin{equation*}
\rho(f_0,f_1,f_2) = - \dfrac{1}{24} \dfrac{1}{2 \pi i} \int_0^{2\pi} f_0f_1'f_2'' \, d \t \, .
\end{equation*}
\end{cor}

The symbol map $\sg_q: C^\infty(SU_q(2)) \ra C^\infty(S^1)$ induces a short exact sequence
\begin{equation*}
0 \longrightarrow \I_q \longrightarrow C^\infty(SU_q(2)) \build \longrightarrow_{}^{\sg_q} C^{\infty} (S^1) \longrightarrow 0 \, .
\end{equation*}
Let $A \in M_N(\Omega^1(C^\infty(SU_q(2))))$ an arbitrary $1$-form. There are $C^\infty(SU_q(2))$-valued $N \times N$ matrices $a^0,\ldots, a^n$ and $b^0, \ldots, b^n$ such that 
\begin{equation*}
A = \sum_i a^idb^i
\end{equation*}
Let $P_q : C^\infty(SU_q(2)) \ra \I_q$ denote the projection onto $\I_q = \ker \sg_q$. We define 
\begin{equation*}
a^i_1 = P_q(a^i), \, b^i_1 = P_q(b^i) \quad \fl i=1, \ldots, n
\end{equation*}
\begin{equation*}
A_1 = \sum_i a^i_1db^i_1 \quad \text{and} \quad A_2 = A -A_1
\end{equation*}

\begin{thm} \label{thm_4.3}
Let $A \in M_N(\Omega^1(C^\infty(SU_q(2))))$ a hermitian $1$-form and $A = A_1 + A_2$ its decomposition. The Chern-Simons action \ref{CS} of $A$ adopts the following form:
\begin{equation*}
S_{CS}(A) = S_{CS}(A_1) - 2 \pi k \phi_1( A_2) 
\end{equation*}
Thus the action is linear in $A_2$ because $\phi_1$ is a $1$-cochain.
\end{thm}

\begin{proof}
Due to theorem \ref{thm_3.2} and \ref{local_index_q} $\phi_3 = b \phi_2$. By corollary \ref{cor_4.2} $\phi_2$ is given as the pullback of a $2$-cochain on $C^\infty(S^1)$ by $\sg_q$. From this, it follows immediately that $\phi_3(a^0da^1da^2da^3) = 0$, with $a^i \in C^\infty(SU_q(2))$, if at least one of the $a^i$ is contained in $\I_q$. From this observation we obtain
\begin{align*}
 S_{CS}(A) =& 6 \pi k \phi_3 \left( A dA + \dfrac{2}{3} A^3 \right) - 2 \pi k \phi_1 (A) \\
=& 6 \pi k \phi_3 \left( (A_1+A_2) d(A_1+A_2) + \dfrac{2}{3} (A_1+A_2)^3 \right) - 2 \pi k \phi_1 (A_1 + A_2) \\
=& 6 \pi k \phi_3 \left( A_1 dA_1 +\dfrac{2}{3} A_1^3 \right) - 2 \pi k \phi_1 (A_1) - 2 \pi k \phi_1(A_2)  \\
=& S_{CS}(A_1) - 2 \pi k \phi_1(A_2) 
\end{align*}
\end{proof}
Let us compute the Chern-Simons action explicitly. Let $A \in M_N(\Omega^1(C^\infty(SU_q(2))))$ be an arbitrary $1$-form. Since we are only interested in the $A_1$-part of the $1$-form, we can assume that $A = \sum_{i} a_i db_i$ with
\begin{equation*} 
\begin{array}{ll}
a_i = \sum_{k \geq 0} \l^i_k \a^k + \sum_{k > 0} \l^i_{-k} \a^{\ast k} & b_i = \sum_{k \geq 0} \mu^i_k \a^k + \sum_{k > 0} \mu^i_{-k} \a^{\ast k} \, ,
\end{array}
\end{equation*}
where $(\l^i_k)_k$ and $(\mu^i_k)_k$ is a sequence of complex $N \times N$-matrices of rapid decay. We are only interested in hermitian $1$-forms. Hence, we symmetrise the expression 
\begin{equation*}
\tilde{A} = \dfrac{1}{2} \left( A + A^\ast \right) = \dfrac{1}{2} \sum_{i} a_i db_i + b_i^{ \ast} da_i^{ \ast} -  d( b_i^{ \ast} a_i^{ \ast} ) \, .
\end{equation*}

\begin{thm} \label{chern_simons_q}
Let $\tilde {A}$ the hermitian $1$-form defined above and
\begin{align*}
\Re_{kl} =& \dfrac{1}{2} \sum_{i} \tr \left( \l^i_{-k} \mu^i_{l} + \mu^{i \ast}_{-l} \l^{i \ast}_{k} \right) \\
\Im_{kl} =& \dfrac{1}{2} \sum_{i} \tr \left( \l^i_{-k} \mu^i_{l} - \mu^{i \ast}_{-l} \l^{i \ast}_{k} \right) \, .
\end{align*}
Then the summands of the Chern-Simons action of $\tilde A$ are 
\begin{align*}
\phi_3( \tilde{A} d \tilde A + \tilde{A}^3 ) =& - \dfrac{1}{12} \sum_{k_i \in \ZH} k_2k_3k_4 \Im_{k_1k_2} \Re_{k_3k_4} \d_{k_2-k_1+k_4-k_3}  \\
&+ \dfrac{1}{18}\sum_{k_i \in \ZH} k_2 k_4 k_6 \Im_{k_1k_2}\Im_{k_3k_4}\Im_{k_5k_6} \d_{k_2-k_1+k_4-k_3+k_6-k_5} 
\end{align*}
with Kronecker delta $\d_k$. The chochain $\phi_1$ has different shapes for the two cases:\\
\textbf{For $0 < q < 1$}
\begin{align*}
\phi_1(\tilde A ) =&  \chi(\tilde A) + b \vp_0 ( \tilde A) + B \vp_2( \tilde A) \\
= & -2 \sum_{k \in \ZH} k \Im_{kk}\, F_k(q) + (1 +i) \sum_{k \in \ZH} k^2 \Re_{kk} + \sum_{k \in \ZH}\sgn(k) k^2  \Im_{kk} \\
& - 2 \sum_{k \in \ZH} \sgn(k)\, \Im_{k k}\, H_{\mid k \mid}(q) - \dfrac{1}{12} \sum_{k \in \ZH } k^3 \Im_{k k} 
\end{align*}
with
\begin{align*}
F_k(q) =& \sum_{x \in \NH}^\infty \left(\prod_{j= 1}^{k} \left(1 - q^{2(j+x)} \right) -1\right) \\
H_k(q) =& \vp_0 \left( \lcom \a^k, \a^{\ast k} \rcom \right)
\end{align*}
\textbf{For $q	= 0$}
\begin{align*}
\phi_1( \tilde A ) =& \quad \tau_1(\tilde A) + b \vp_0 (\tilde A ) + B \vp_2 (\tilde A) \\
=& - 2 \sum_{k \in \ZH \setminus \set{0}}  \sgn(k)\, \Im_{kk} \, \sum^{\mid k-1 \mid}_{j=1} \rho(j) - \dfrac{1}{12} \sum_{k \in \ZH} k^3 \Im_{kk} 
\end{align*}
\end{thm}

\begin{proof}
Let us compute the chochain $\phi_1$ for the case $ 0 < q < 1$. We start with $\chi( \tilde A)$. Since $\chi = \int$ is a closed trace, we obtain $\chi(d(b_i^\ast a_i^\ast)) =0$. Hence,
\begin{align*}
\chi(\tilde A ) =& \dfrac{1}{2} \sum_i \chi( a_i db_i + b_i^\ast da_i^\ast ) \\
=& \quad \dfrac{1}{2} \sum_i \chi \left( \sum_{k,l \geq 0 } \l_k^i \mu_l^i \a^k d\a^l \right) + \dfrac{1}{2} \sum_i \chi \left( \sum_{k \geq 0, l > 0 } \l_k^i \mu_{-l}^i \a^k d\a^{\ast l } \right) \\
& + \dfrac{1}{2} \sum_i \chi \left( \sum_{k>0 , l \geq 0 } \l_{-k}^i \mu_l^i \a^{\ast k} d\a^l \right) + \dfrac{1}{2} \sum_i \chi \left( \sum_{k,l > 0 } \l_{-k}^i \mu_{-l}^i \a^{\ast k} d\a^{\ast l} \right) \\
& + \dfrac{1}{2} \sum_i \chi \left( \sum_{l,k \geq 0 }  \mu_l^{i \ast}  \l_k^{i \ast} \a^{\ast l} d\a^{\ast k} \right) + \dfrac{1}{2} \sum_i \chi \left( \sum_{l \geq 0, k > 0 }  \mu_{l}^{i \ast} \l_{-k}^{i \ast}  \a^{\ast l} d\a^{ k } \right) \\
& + \dfrac{1}{2} \sum_i \chi \left( \sum_{l>0 , k \geq 0 }  \mu_{-l}^{i \ast} \l_{k}^{i \ast} \a^{ l} d\a^{\ast k} \right) + \dfrac{1}{2} \sum_i \chi \left( \sum_{l,k > 0 } \mu_{-l}^{i \ast} \l_{-k}^{i \ast} \a^{ l} d\a^{ k} \right) 
\end{align*}
\begin{align*}
=&  \quad \sum_{k,l \geq 0 } \dfrac{1}{2} \sum_i \tr \left( \l_k^i \mu_l^i \right) \chi \left( \a^k d\a^l \right) + \sum_{k \geq 0, l > 0 } \dfrac{1}{2} \sum_i \tr \left( \l_k^i \mu_{-l}^i \right) \chi \left( \a^k d\a^{\ast l } \right) \\
& +\sum_{k>0 , l \geq 0 } \dfrac{1}{2} \sum_i \tr \left( \l_{-k}^i \mu_l^i \right) \chi \left( \a^{\ast k} d\a^l \right) +  \sum_{k,l > 0 } \dfrac{1}{2} \sum_i \tr \left( \l_{-k}^{i}  \mu_{-l}^{i} \right) \chi \left( \a^{\ast k} d\a^{\ast l } \right) \\
& + \sum_{l,k \geq 0 } \dfrac{1}{2} \sum_i \tr \left( \mu_l^{i \ast} \l_k^{i \ast} \right) \chi \left( \a^{\ast l} d\a^{\ast k} \right) +  \sum_{l \geq 0, k > 0 } \dfrac{1}{2} \sum_i \tr \left( \mu_{l}^{i \ast} \l_{-k}^{i \ast} \right) \chi \left(\a^{\ast l} d\a^{k } \right) \\
& + \sum_{l>0 , k \geq 0 } \dfrac{1}{2} \sum_i \tr \left( \mu_{-l}^{i \ast} \l_{k}^{i \ast} \right) \chi \left( \a^{l} d\a^{\ast k} \right) + \sum_{l,k > 0 }  \dfrac{1}{2} \sum_i \tr \left( \mu_{-l}^{i \ast} \l_{-k}^{i \ast} \right) \chi \left( \a^{ l} d\a^{ k} \right) \\
\end{align*}
From the equations \eqref{eq_3.6},\eqref{eq_3.7} and \eqref{eq_3.8} one gets for $k,l \in \ZH$
\begin{align*}
\chi \left( \a^k d \a^l \right)=& \quad \chi \left( \a^k \partial(\a^l) \right) + \chi \left( \a^k \dfrac{1}{2} \sg(\a^l)'' d\t^2 \right) \\
=& -l \tau_0 \left( r_- \left( \a^k  \a^l \right)^0 \right) + \dfrac{1}{4 \pi i} \int_{0}^{2 \pi} e^{ik \t} \left(e^{il \t} \right)'' \, d\t \\
=& -l \tau_0 \left( r_- \left( \a^k  \a^l \right) \right) \d_{k+l} -  \dfrac{l^2}{2 i} \d_{k+l}
\end{align*}
with $\a^l = \a^{\ast (-l)}$ for $l < 0$ and the Kronecker delta $\d_m$. Therefore, the above expression reduces to
\begin{align*}
\chi(\tilde A) =& \quad \sum_{k > 0 } \sum_i \tr \left( \l_k^i \mu_{-k}^i \right) \left( k \tau_0(r_-(\a^k \a^{\ast k})) - \dfrac{k^2}{2 i}\right) \\
 &+  \sum_{k > 0} \sum_i \tr \left( \l_{-k}^i \mu_k^i \right) \left( -k \tau_0(r_-(\a^{\ast k} \a^{ k})) - \dfrac{k^2}{2 i} \right) \\
& + \sum_{k > 0 } \sum_i \tr \left( \mu_{k}^{i \ast} \l_{-k}^{i \ast} \right) \left( -k \tau_0(r_-(\a^{\ast k} \a^{ k})) - \dfrac{k^2}{2i} \right) \\
& + \sum_{k>0} \sum_i \tr \left( \mu_{-k}^{i \ast} \l_{k}^{i \ast} \right) \left( k \tau_0(r_-(\a^k \a^{\ast k})) - \dfrac{k^2}{2 i} \right) \\
= & \quad \sum_{k > 0 } \sum_i k \tr \left( \l_k^i \mu_{-k}^i \right) \tau_0(r_-(\a^k \a^{\ast k})) 
 -  k \tr \left( \l_{-k}^i \mu_k^i \right) \tau_0(r_-(\a^{\ast k} \a^{ k}))  \\
& - \sum_{k > 0 }  \sum_i k \tr \left( \mu_{k}^{i \ast} \l_{-k}^{i \ast} \right) \tau_0(r_-(\a^{\ast k} \a^{ k}))  - k  \tr \left( \mu_{-k}^{i \ast} \l_{k}^{i \ast} \right) \tau_0(r_-(\a^k \a^{\ast k})) \\
& + i \sum_{k \in \ZH} k^2 \Re_{kk} 
\end{align*}
Using formula (8.5) in \cite{connes_sphere} one gets
\begin{equation*}
\tau_0(r_-(\a^{\ast k} \a^{ k})) = \tau_0(r_-(\a^{k} \a^{ \ast k})) - \dfrac{1}{2 \pi i} \int_0^{2 \pi} \sg(\a^{ \ast k}) d \sg(a^{k} ) = \tau_0(r_-(\a^{k} \a^{ \ast k})) - k \\
\end{equation*}
Inserting this in the formulae above, one obtains
\begin{align*}
\chi(\tilde A) = & \quad \sum_{k > 0 } \sum_i k \tr \left( \l_k^i \mu_{-k}^i \right) \tau_0(r_-(\a^k \a^{\ast k})) 
 -  k \tr \left( \l_{-k}^i \mu_k^i \right) \tau_0(r_-(\a^{k} \a^{\ast k}))  \\
& - \sum_{k > 0 } \sum_i k \tr \left( \mu_{k}^{i \ast} \l_{-k}^{i \ast} \right) \tau_0(r_-(\a^{k} \a^{\ast k}))  - k \tr \left( \mu_{-k}^{i \ast} \l_{k}^{i \ast} \right) \tau_0(r_-(\a^{k} \a^{\ast k})) \\
& + i \sum_{k \in \ZH} k^2 \Re_{kk} + \sum_{k > 0} \sum_i k^2 \tr \left( \l_{-k}^i \mu_k^i \right) + k^2 \tr \left( \mu_{k}^{i \ast} \l_{-k}^{i \ast} \right) \\
=& - 2 \sum_{k \in \ZH} k \, \Im_{kk} \tau_0(r_-(\a^k \a^{\ast k})) + i \sum_{k \in \ZH} k^2 \Re_{kk} \\
& + \sum_{k > 0} \sum_i k^2 \tr \left( \l_{-k}^i \mu_k^i \right) + k^2 \tr \left( \mu_{k}^{i \ast} \l_{-k}^{i \ast} \right) \\
=& - 2 \sum_{k \in \ZH} k \, \Im_{kk} \tau_0(r_-(\a^k \a^{\ast k})) + i \sum_{k \in \ZH} k^2 \Re_{kk} \\
& + \sum_{k > 0} k^2 \left( \Re_{kk} + \Im_{kk} + \Re_{-k -k} - \Im_{-k -k} \right) \\
=& -2 \sum_{k \in \ZH} k \, \Im_{kk} \tau_0(r_-(\a^k \a^{\ast k})) + (1 + i) \sum_{k \in \ZH} k^2 \Re_{kk} + \sum_{k \in \ZH } \sgn(k) k^2 \Im_{kk}
\end{align*}
By the defining commutation relations of $SU_q(2)$ one obtains immediately
\begin{equation*} 
\a^k \a^{\ast k} = \prod_{j=1}^{k} \left(1 - q^{2j} \b \b^\ast \right) \, .
\end{equation*}
By the definition of the representations $\pi_\pm$ \eqref{eq_3.4}, the operator $r_-(\b \b^\ast)$ acts as a diagonal operator on $\ell^2(\NH)$ with $\pi_- (\b\b^\ast) \varepsilon_x = q^{2x} \varepsilon_x$, for $x \in \NH$. Hence from the definitions \eqref{tau_0} and \eqref{tau_1} of $\tau_0$ and $\tau_1$ we obtain
\begin{align*}
&\tau_0(r_-(\a^k \a^{\ast k})) \\
&= \lim_{N \ra \infty} \tr_N(\pi_-(\a^k \a^{\ast k})) - N \tau_1(\a^k \a^{\ast k}) \\
&= \lim_{N \ra \infty} \tr_N\left(\pi_-\left(\prod_{j=1}^{k} \left(1 - q^{2j} \b \b^\ast \right) \right) \right) - N  \\
&= \lim_{N \ra \infty} \sum_{x =1}^N  \prod_{j=1}^k \left( \left( 1-q^{2(j+x)} \right) -1 \right) = F_k(q) \, .
\end{align*}
Using lemma \ref{symbol} and the definition of the $B$-operator one gets
\begin{align*}
B \vp_2(\tilde A) =&  \dfrac{1}{2}\sum_i \vp_2 \left(da_idb_i - db_ida_i + db_i^{\ast}da_i^{\ast} - da_i^{\ast}db_i^{\ast} \right) \\
=& -\dfrac{1}{48} \dfrac{1}{2 \pi i} \sum_i \int_0^{2 \pi} \tr \left( \sg_q(a_i)' \sg_q(b_i)'' - \sg_q(b_i)' \sg_q(a_i)'' \right) \, d \t \\
& -\dfrac{1}{48} \dfrac{1}{2 \pi i} \sum_i \int_0^{2 \pi}  \tr \left( \sg_q(b_i^{\ast})' \sg_q(a_i^{\ast})'' - \sg_q(a_i^{\ast})' \sg_q(b_i^{ \ast})'' \right) \, d \t 
\end{align*}
By means of  $\tr \left( \sg_q(b_i)' \sg_q(a_i)'' \right) = \tr \left( \sg_q(a_i)'' \sg_q(b_i)' \right) $ and partial integration one achieves
\begin{equation*}
\int_0^{2 \pi} \tr \left( \sg_q(a_i)' \sg_q(b_i)'' - \sg_q(b_i)' \sg_q(a_i)'' \right) \, d \t =  2 \int_0^{2 \pi} \tr \left( \sg_q(a_i)' \sg_q(b_i)'' \right) \, d\t
\end{equation*}
and in an analogue way for the second integral. Hence,
\begin{align*}
B \vp_2(\tilde A) =& -\dfrac{1}{24} \dfrac{1}{2 \pi i} \sum_i \int_0^{2\pi}  \tr \left( \sg_q(a_i)' \sg_q(b_i)'' + \sg_q(b_i^{\ast})' \sg_q(a_i^{\ast})''\right) \, d \t \\
=& -\dfrac{1}{24} \dfrac{1}{2 \pi i} \sum_{k,l \in \ZH} \sum_i \\
& \times \int_0^{2\pi} \tr \left( \l_k^i \mu_l^i \right) \left(e^{ik}\right)' \left(e^{il}\right)'' +  \tr \left( \mu_k^{i \ast} \l_l^{i \ast} \right) \left(e^{-ik}\right)' \left(e^{-il}\right)'' \, d\t \\
=& \quad \dfrac{1}{24}  \sum_{k,l \in \ZH} \sum_i kl^2 \tr \left( \l_k^i \mu_l^i \right) \d_{k+l} - kl^2 \tr \left( \mu_k^{i \ast} \l_l^{i \ast} \right) \d_{k+l} \\
=& \quad \dfrac{1}{24}  \sum_{k \in \ZH} k^3 \sum_i \tr \left( \l_k^i \mu_{-k}^i  - \mu_k^{i \ast} \l_{-k}^{i \ast} \right) \\
=&  - \dfrac{1}{12} \sum_{k \in \ZH} k^3 \Im_{kk} \,
\end{align*}
Finally, the the cochain $b\vp_0$. Since of analogous grading properties as in the case of the cycle $\chi$ one obtains

\begin{align*}
b \vp_0 (\tilde A ) =& b \vp_0 \left( a_i db_i + b_i^\ast da_i^\ast -d(a_i^\ast b_i^\ast) \right) \\
=& \vp_0 \left( \lcom a_i, b_i \rcom \right) + \vp_0 \left( \lcom b_i^\ast, a_i^\ast \rcom \right) - \vp_0 \left(\lcom 1, a_i^\ast b_i^\ast \rcom \right) \\
=& \quad  \sum_{k > 0 } \sum_i \tr \left( \l_k^i \mu_{-k}^i \right) \vp_0 \left( \lcom \a^k , \a^{\ast k} \rcom \right)
 +  \tr \left( \l_{-k}^i \mu_k^i \right) \vp_0 \left( \lcom \a^{\ast k} , \a^{ k} \rcom \right) \\
& + \sum_{k > 0 } \sum_i \tr \left( \mu_{k}^{i \ast} \l_{-k}^{i \ast} \right) \vp_0 \left( \lcom \a^{\ast k} , \a^{ k} \rcom \right) + \tr \left( \mu_{-k}^{i \ast} \l_{k}^{i \ast} \right) \vp_0 \left( \lcom \a^k ,\a^{\ast k} \rcom \right) \\
=& \quad \sum_{k > 0} \sum_i \left( \tr \left( \l_k^i \mu_{-k}^i \right) - \tr \left( \mu_{k}^{i \ast} \l_{-k}^{i \ast} \right) \right) H_k(q)\\
&- \sum_{k > 0} \sum_i \left( \tr \left( \l_{-k}^i \mu_k^i \right) - \tr \left( \mu_{-k}^{i \ast} \l_{k}^{i \ast} \right) \right) H_k(q) \\
=& - 2 \sum_{k \in \ZH} \sgn(k)\, \Im_{kk} \, H_{\mid k \mid}(q)
\end{align*}
The computation of the cochain $\phi_1$ for the case $q = 0$ is much simpler and be omitted. The same holds true for the computation of the cochain $\phi_3$. In the computation of this cochain the different cases need not be distinguished, and the calculation can be done easily if one uses corollary \ref{cor_4.2} and lemma \ref{symbol}.
\end{proof}

\begin{rem}
We want to point out two important aspects of the special shape of the Chern-Simons action computed in theorem \ref{chern_simons_q}. Firstly, the action depends on the deformation parameter $q$, at least the linear $\phi_1$-part of the action. Hence, the critical points of the action is shifted differently for different $q$'s. \\
Secondly, the action depends only on the values $\Re_{kl}$ and $\Im_{kl}$, with $k,l \in \ZH$, and not on the parameters $\l^i_k$ and $\mu_l^i$ itself. Using this property of the action we can define a measure for the path integral consisting of an infinite product of the Lebesgue measures $d\Re_{kl}$ and $d \Im_{kl}$. This will be carried out in detail in the next section. 
\end{rem}

\section{path integral}

In this section we study the drawbacks and opportunities of an explicit computation of the path integral
\begin{equation*}
Z(k) = \int D \lcom A \rcom \, e^{i S_{CS}(A)} \, ,
\end{equation*}
where integration is performed over all hermitian $\Omega^1(C^\infty(SU_q(2)))$-valued $N \times N$ matrices, modulo gauge transformations. Firstly, we give an outline of the gauge breaking mechanism by Faddeev-Popov. For a detailed discussion we refer the reader to G. B. Follands book \cite{folland}.\\
Then we have to make sense of the path integral in the noncommutative setting. Theorems \ref{thm_4.3} and \ref{chern_simons_q} are crucial tools in this section, not only for the definition of the path integral on $SU_q(2)$ but also for a -- at least conceptual -- computation of it. Finally, we point out the problems of an explicit computation by means of a loop expansion series, i. e. a Taylor expansion series of the partition function $Z(k)$ in the variable $k^{-1}$. These problems particularly concern the linear shift of the Chern-Simons action caused by the $\phi_1$-chochain in the index formula.

\medskip

\subsection{Gauge-breaking}

In the computation of the path integral one does not integrate over all hermitian $1$-forms, but only over equivalence classes modulo gauge transformations. In order to achieve this we restrict the space of all $1$-forms to a subset where all configurations modulo gauge transformations are only counted once, i. e. we fix the gauge or break the symmetry. The usual way of doing this is a device due to Faddeev and Popov explained thoroughly in \cite{folland}. To explain the idea, let us consider a similar but much simpler situation. Let $ \set{\sg_t : \, t \in \RH} $ be a one-parameter group of measure-preserving diffeomorphisms on $\RH^n$ whose orbits are (generically) unbounded, and suppose $F$ is a function which is invariant under these transformations. We wish to extract a finite and meaningful quantity from the divergent integral $\int F(x) \, d^n x $. One possibility is to find a hypersurface $M$ that is a cross-section for the orbits and consider instead the integral $\int_M F(x) \, d \Sigma (x) $ where $d \Sigma $ is the surface measure of $M$. This quantity, however, depends on the choice of $M$. A related procedure is to find a smooth function $h$ such that $M = h^{-1}(0)$ and consider the integral $\int F(x)\d(h(x)) \, d^n x$. This quantity depends on the choice of both $M$ and $h$, for one must take into account the behaviour of the delta function under a change of variable. (The basic formula is this: if $ \phi $ is a smooth function on $\RH$ and there is a unique $t_0$ such that $\phi(t_0)=0$, then $\d(t-t_0) = \phi'(t_0) \d(\phi(t))$.) \\
A better idea is to incorporate the appropriate change-of-measure factor into the integral. Specifically, with $M$ and $h$ as above, let 
\begin{equation*}
\Delta (x) = \dfrac{d}{dt} \left[ h(\sigma_t(x)) \right] _{t=t(x)} \, ,
\end{equation*}
where $t(x)$ is the unique number such that $ h(\sg_{t(x)} (x)) = 0 $, and insert the factor
\begin{equation*}
 1 = \int \d(u) du = \int \d( h(\sg_{t(x)})) \Delta (x) \, dt
\end{equation*} 
into the integral $\int F(x) \, d^n x$ to obtain (informally speaking)
\begin{equation*}
\int F(x) \, d^n x = \int F(x) \d(h(\sg_{t(x)})) \d (x) \, dt d^n x.
\end{equation*} 
Now, $F$ and the measure are assumed to be invariant under the transformations $\sg_t$, and it is easily checked that this holds true for $\Delta(x)$ too, see \cite{folland} chapter seven. Hence we can make the substitution $x= \sg_{-t}(y)$ to obtain
\begin{equation*}
\int F(x)d^n x = \left[ \int \, dt \right] \left[ \int F(y) \d(h(y)) \Delta(y) \, d^n y \right].
\end{equation*}
The divergence has now been isolated as the infinite factor $\int \, dt$. The remaining $y$-integral is the quantity we have been seeking: it is finite provided the restriction of $F$ to $M$ decays suitably at infinity, and one can verify that it is independent of the choice of $M$ and $\phi$. (On the informal level, one can observe that $\int \, dt$ is merely the volume of the transformation group, which does not depend on $M$, $\phi$ or $F$, so that its removal should yield an invariant quantity too. This, of course, is the reasoning employed in the functional integral situation, where everything is somewhat illdefined.) The same idea works for multi-parameter groups of transformations; the factor $\Delta(x)$ there is an appropriate Jacobian determinant. \\

With this prelude in mind we can start with our rearrangement of the path integral. Our gauge group is the group $\mathcal{U} M_N (C^\infty(SU_q(2)) )$ of unitary elements, and the gauge transformation is $A \mapsto A^{u} = u A u^{\ast} + u d u^{\ast}$ for a unitary element $u \in \mathcal{U} M_N (C^\infty(SU_q(2)) )$. We are choosing a gauge fixing $h$ of the following form. \\
Let $\mathfrak{X} = \lbrace x_n: \, n \geq 0 \rbrace$ and $\CH \lcom \mathfrak{X} \rcom$ the formal ring of finite polynomials with complex coefficients in the variables $x_n$. For $h \in \CH \lcom \mathfrak{X} \rcom$ there is an integer $k$ such that $h $ depends only on $x_0, \ldots, x_k$ and
\begin{equation*}
h(x_0, \ldots, x_k) = \sum a^{m_0, \ldots, m_j}_{n_0, \ldots, n_j} x_{n_0}^{m_0} \cdots  x_{n_j}^{m_j}
\end{equation*} 
where only finitely many coefficients $a^{m_0, \ldots, m_j}_{n_0, \ldots, n_j}$ are different from zero. Given such a gauge fixing $h \in \CH \lcom \mathfrak{X} \rcom $ we define for any 1-form $A = \sum_{i}a_idb_i$ 
\begin{align*}
h (A) =& \sum_{i} h (a_idb_i) \\ 
=& \sum_{i} h \left( a_i \d (b_i), \ldots, \d^k ( a_i \d( b_i)) \right) \\
=& \sum_{i} \sum a^{m_0, \ldots, m_j}_{n_0, \ldots, n_j} \left( \d^{n_0} (a_i \d (b_i) ) \right)^{m_0} \cdots \left( \delta^{n_j} (a_i \d( b_i) )\right)^{m_j} 
\end{align*}

\begin{rem}
Usually, in the commutative case, i. e. if one works with a $3$-dimensional spin manifold $(M, \S)$, one chooses always Lorenz gauge. Every local map $(U, \chi)$ consisting of an open subset $U \subset M$ and a diffeomorphism $\chi: U \ra \chi(U) \subset \RH^3$ provides us with a local trivialization $\pi^{-1}(U) \cong U \times \RH^3$ of the tangent bundle $TM$, where $\pi: TM \ra M$ denotes the canonical projection. Let $A \in M_N(\Omega^1(M))$ be a matrix of differential $1$-forms on $M$. Due to the local trivialization the restriction $A_U$ of $A$ on the open subset $U$ decomposes into $ A = f_\mu dx^\mu $ for smooth functions $f_\mu: U \ra M_N(\CH)$, with $\mu = 1,2,3$. Lorenz gauge condition can be expressed, at least locally, by $ \partial/\partial x_1 \,  f_1 + \partial/\partial x_2 \,  f_3 + \partial/\partial x_3 \,  f_3 = \omega $, for a smooth, matrix valued function $\omega \in C^\infty(U) \ot M_N(\CH)$ on $U$.\\
In noncommutative geometry neither local maps nor a related decomposition of $1$-forms $A \in M_N(\Omega^1(\A))$ is available in general. There is one exception: the noncommutative $3$-torus $C^\infty(\TH_\Th^3)$ whose Chern-Simons theory is studied in \cite{pfante_torus}. As noncommutative generalization of the $3$-torus $\TH^3$ -- whose tangent bundle $T \TH^3$ is globally trivial, i. e. $T \TH^3 \cong \TH^3 \times \RH^3$ --  $1$-forms $A \in M_N(\Omega^1(C^\infty(\TH_\Th^3))$ always decompose as in the commutative case, see \cite{iochum} for details. Therefore, Lorenz gauge can be chosen and is used in \cite{pfante_torus} to compute the path integral for the noncommutative $3$-torus.\\
For $SU_q(2)$ such a decomposition of $1$-forms is not available at all due to the lack of local maps, which are necessary in the commutative case $S^3 \cong SU(3)$ because the tangent bundle $T S^3$ is not longer globally trivial. Hence we cannot use Lorenzian gauge for $SU_q(2)$ and we were forced introducing a different gauge. But Lorenz gauge is not the canonical one. Any other gauge fixing is as good as the Lorenzian one. Different gauge fixings are studied by P. Gaigg, W. Kummer, and M. Schweda in \cite{schweeder}.
\end{rem}

Denote by $u(A)$ the unitary element such that $h( A^{u(A)} ) - \omega = 0 $, and let $\tilde{A} = A ^{u(A)}$. The geometrical interpretation of $\tilde{A}$ is the following: For a given configuration we follow the orbit $A^{u}$ until we arrive at the configuration given by the gauge fixing.  As before, we write 
\begin{equation} \label{eq_5.1}
1 = \int Du \d(u) = \int Du \d \left( h ( \tilde{A}^{u} ) - \omega \right) \det \left( \left. \d h(\tilde{A}^{u})/ \d u \right|_{ u = 1}  \right) .
\end{equation} 
The integrals here are functional integrals over the space $\mathcal{U} M_N (C^\infty(SU_q(2)) )$ of all unitary elements, the delta-function represents the point mass at the identity $ 1 \in \mathcal{U}M_N (C^\infty(SU_q(2)) )$ and the differential $\left. \d h(\tilde{A}^{u})/ \d u \right|_{ u = 1} $ is the formal \textit{functional derivative} of $ h( \tilde{A}^{u} )$ with respect to $ u $ at the point $ u = 1 $. Moreover, \lgans$ \det $" denotes the functional determinant, the counterpart of the Jacobian determinant in our infinite dimensional setting. \\

For the following calculation we rename $\tilde{A}$ to $A$ and we insert \eqref{eq_5.1} in the path integral, obtaining
\begin{equation*}
\int DA \exp(i S_{CS}(A)) = \iint Du DA \, \exp \left(i S_{CS}(A) \right) \delta(h(A^{u})- \omega)\det \left( \left. \d h(A^{u})/ \d u \right|_{ u = 1}  \right) \, .
\end{equation*} 
The \lgans Lebesgue measure" $DA$ is gauge invariant because a gauge transformation $ A \mapsto u A u^{\ast} + u d u^{\ast} $ results in a conjugation and a simple translation in the space of fields, and both operations leave this space unchanged. Additionally, we have
\begin{equation*}
\exp \left(iS_{CS}(A^u)\right) = \exp\left(i S_{CS}(A) + 2 \pi i \ind \left(P u P \right) \right) = \exp \left( i S_{CS}(A) \right)\, .
\end{equation*} 
Hence the substitution $A^{u^{\ast}}$ for $A$ turns the above integrand into an expression that is independent of $u$, so that
\begin{equation*}
\int DA \, \exp \left(i S_{CS}(A) \right) = \int Du \int DA \, \exp \left( i S_{CS}(A) \right) \d(h(A) - \omega) \det \left( \left. \d h(A^{u})/ \d u \right|_{ u = 1}  \right)
\end{equation*} 
The integral $\int Du$ (the volume of the gauge group) is an infinite constant that can be ignored, and the integral that is left has some hope of being meaningful; this is the analogue of the finite-dimensional result that we derived above.\\
To proceed further, we put the arbitrary element $\omega \in M_N(C^\infty(SU_q(2)))$ to use. Since the above identity is valid for each $\omega$, it remains valid if we take a weighted average over different $\omega$'s. We can then get rid of the $\d$-function by multiplying both sides with 
\begin{equation*}
N(a)^{-1} =   \int D \omega \, \exp \left( \ncint -\frac{i}{2a} \omega^2  \AD^{-3} \right) \,
\end{equation*}
a noncommutative analogue of the Gaussian integral, where \lgans$N(a)"$ means that the normalization of the Gaussian integral will depend on the parameter $a$. However, again it is a constant (infinite) normalization factor independent of any dynamics. Moreover, $\ncint$ denotes the noncommutative replacement of the integral -- the Wodzicki residue of theorem \ref{integrals_su_q_2}. As a result we get  
\begin{align*} 
&\int DA \, \exp \left( i S_{CS}(A) \right)  = N(a) \int Du \\
& \times \iint DA D\omega \, \exp\left( iS_{CS}(A) -\frac{i}{2a} \, \ncint  \omega^{2} \AD^{-3}\right) \delta \left( h(A) - \omega \right) \det \left( \left. \d h(A^{u})/ \d u \right|_{ u = 1}  \right) \, . 
\end{align*} 
Performing the integration over $\omega$ we obtain, up to an infinite constant, 
\begin{equation} \label{eq_5.2} 
\int DA \, \exp\left(i S_{CS}(A) - \frac{i}{2a} \, \ncint h(A)^2 \AD^{-3}\right) \det \left( \left. \d h(A^{u})/ \d u \right|_{ u = 1} \right) \, . 
\end{equation}
Finally, we represent the functional determinant as a Gaussian integral over Grassmann variables due to the following lemma which describes the finite dimensional case.
 
\begin{lem}
\begin{equation*}
\iint \left[ d\zeta^{\ast} d\zeta \right] \, \exp(-\zeta^{\ast} A \zeta)= \det A \, ,
\end{equation*}
with the notational conventions
\begin{equation*}
\zeta^{\ast} A \zeta = \sum_{j,k} \zeta^{\ast}_j A_{jk} \zeta_k, \quad \iint \left[ d\zeta^{\ast} d\zeta \right] = \iint \cdots \iint d\zeta^{\ast}_n d\zeta_n \cdots d\zeta^{\ast}_1 d\zeta_1 \, ,
\end{equation*}
for complex Grassmann variables $\zeta_j$, $\zeta_j^\ast$ $(j=1, \ldots, n)$, and a complex $n \times n$ matrix $A$. 
\end{lem}

Let us denote by $\G$ the space of complex Grassmann numbers generated by the pairwise anticommuting Grassman variables $\set{\zeta_j, \zeta_j^\ast}_{j \in \NH}$. We call $\G \ot M_N(\Br_q)$ the space of \textit{ghost-fields}. The noncommutative integral, the Wodzicki-residue, generalizes in a canonical fashion on $\G \ot M_N(\Br_q)$, via
\begin{equation*}
\ncint v \ot a \AD^{-3} =  \left( \ncint a \AD^{-3} \right) v \quad \fl a \in M_N(\Br_q), \, v \in \G \, .
\end{equation*}
Equipped with these preliminaries we rewrite the functional determinant
\begin{equation*} 
\int D \oc Dc \, \exp\left( - \ncint \oc \left. \d h(A^{u})/ \d u\right|_{ u = 1} c \AD^{-3} \right) = \det \left( \left. \d h(A^{u})/ \d u \right|_{ u = 1}  \right)
\end{equation*} 
for two ghost-fields $ \oc , c \in \G \ot M_N(\Br_q)$. Inserting this expression in \eqref{eq_5.2} yields
\begin{equation} \label{eq_5.3}
\int D\oc Dc DA \, \exp \left(i S_{CS}(A) - \ncint \left( \frac{i}{2a} h(A)^2  +  \oc \left. \d h(A^{u})/ \d u  \right|_{ u = 1} c \right) \AD^{-3} \right) 
\end{equation}

\begin{rem}
Note that in \eqref{eq_5.3} we integrate over hermitian $1$-forms itself instead of equivalence classes, which is actually the aim doing gauge breaking. Therefore, \eqref{eq_5.3} allows an explicit computation which is performed in the next subsection as far as possible. 
\end{rem}

\medskip

\subsection{Reduction theorem and Conclusions}

We work out a simplification of \eqref{eq_5.3} by means of theorem \ref{thm_4.3} and lemma \ref{symbol}. As  a result we yield that the expression \eqref{eq_5.3} depends only on the $A_1$-part of $A$ up to an infinite, but irrelevant constant. Using this simplification we receive a more or less precise meaning of the notion \lgans integrating over all hermitian $1$-forms $A$". \\
Let $A \in M_N(\Omega^1(C^\infty(SU_q(2))))$ be a hermitian matrix of $1$-forms, $A =A_1 + A_2$ its decomposition given by theorem \ref{thm_4.3}. By means of lemma \ref{symbol} and the definition of the gauge condition $h$ we get 
\begin{align*}
\ncint \frac{i}{2a} h(A)^2 \AD^{-3} =& \dfrac{i}{4a \, \pi} \int_{0}^{2 \pi} \sg_q(h(A^2)) \, d\t \\
=&  \dfrac{i}{4a \, \pi} \int_{0}^{2 \pi} \sg_q(h(A_1^2)) \, d\t \\
=& \ncint \frac{i}{2a} h(A_1)^2 \AD^{-3} \,. 
\end{align*}
Let $x \in C^\infty(SU_q(2))$ be a selfadjoint element and $u_t = \exp(itx)$, for $t \in \RH$, a continuous 1-parameter group of unitary elements in $C^\infty(SU_q(2))$. We rewrite the functional derivative $ \left. \delta  h(A^{u})/ \delta u  \right|_{ u = 1} $ as 
\begin{equation*}
\left. \delta h(A^{u})/ \delta u  \right|_{ u = 1} = \lim_{t \rightarrow 0} \left( h(A^{u_t}) - h(A) \right) u_{-t} \, . 
\end{equation*}
In order to apply lemma \ref{symbol} again we define the symbol map $\sg_q: \G \ot M_N(\Br_q) \ra \G$ also for ghost-fields by $\sg_q(v \ot a ) = v \sg_q(a) $ for all $a \in M_N(\Br_q)$ and $v \in \G$.
\begin{align*}
&\ncint \oc \left. \d h(A^{u})/ \d u  \right|_{ u = 1}  c \AD^{-3} \\
=& \ncint \oc \lim_{t \ra 0} \left( h(A^{u_t}) - h(A) \right) u_{-t} \, c  \AD ^{-3}\\
=& \dfrac{1}{2 \pi} \int_0^{2\pi} \sg_q(\oc) \sg_q \left(\lim_{t \rightarrow 0} \left( h(A^{u_t}) - h(A) \right) u_{-t}\right)  \sg_q( \oc) \, d \t \\
=&  \dfrac{1}{2 \pi} \int_0^{2\pi} \sg_q(\oc)\left(\lim_{t \ra 0} \left(  \sg_q(h(A^{u_t})) - \sg_q(h(A)) \right) \sg_q(u_{-t}) \right) \sg_q(c) \, d \t \\
=&  \dfrac{1}{2 \pi} \int_{0}^{2\pi} \sg_q(\oc) \left(\lim_{t \ra 0} \left( \sg_q(h(A_1^{u_t})) - \sg_q(h(A_1)) \right) \sg_q(u_{-t}) \right) \sg_q(c) \, d \t \\
=& \ncint \oc \lim_{t \ra 0} \left( h(A_1^{u_t}) - h(A_1) \right) u_{-t} \,c \AD^{-3}  \\
=& \ncint \oc \left. \d h(\A_1^{u})/ \d u  \right|_{ u = 1}  c  \AD^{-3} 
\end{align*}
where the fourth row follows from the fact that $\sg$ is continuous on $M_N(\Br_q)$. We are ready to state the main result of this section.

\begin{thm} \label{thm_5.2}
Let $A = A_1 + A_2$ a hermitian form in $M_N(\Omega^1(C^\infty(SU_q(2))))$ and its decomposition given by theorem \ref{thm_4.3}. The path integral of the Chern-Simons action $S_{CS}(A)$ over all hermitian 1-forms up to gauge equivalence depends only on the $A_1$-part of the action, i. e. up to an infinite -- but irrelevant -- constant, it is equal to 
\begin{equation*}
\int DA_1 Dc D\oc \, \exp \left( iS_{CS}(A_1) - \ncint\left(  \dfrac{i}{2a} h(A_1)^2 + \left. \oc \d h(A_1^{u})/ \d u  \right|_{ u = 1} c \right) \AD^{-3} \right)
\end{equation*}
for any gauge fixing $h \in \CH \lcom \mathfrak{X} \rcom$.
\end{thm}

\begin{proof}
From \ref{thm_4.3} we obtain $ S_{CS}(A) = S_{CS}(A_1) - 2 \pi k \phi_1(A_2) $ and by the considerations above we can perform the integration in the $A_2$-part for itself, i. e. 
\begin{align*}
&= \int Dc D\oc DA \,  \exp \left( i S_{CS}(A) - \ncint \left( \dfrac{i}{2a} h(A)^2  + \left. c^\ast \d h(A^{u})/ \d u  \right|_{ u = 1} c \right)  \AD^{-3} \right) \\
&= \int Dc D\oc DA_1 DA_2 \, \exp \left( i S_{CS}(A_1) + i S_{CS}(A_2) \right) \\
& \quad \times \exp \left( - \ncint \left( \dfrac{i}{2a} h(A_1)^2 + \left. \oc \d h(A_1^{u})/ \d u  \right|_{ u = 1} \oc  \AD^{-3} \right) \right) \\
&= \int DA_2 \,  \exp \left(- 2 \pi i k \phi_1(A_2) \right) \\
& \quad \times \int Dc D\oc DA_1 \, \exp \left( i S_{CS}(A_1)  - \ncint  \left( \dfrac{i}{2a} h(A_1)^2  + \left. \oc \d h(A_1^{u})/ \d u  \right|_{ u = 1} c \right)  \AD^{-3}  \right)
\end{align*}
Since the cochain $\phi_1$ is linear in $A_2$ the first integral is a Fourier integral over $A_2$ and the statement follows. 
\end{proof}

\begin{rem}
By means of the theorems \ref{chern_simons_q} and \ref{thm_5.2} one might think that the variables $\Im_{kl}$ and $\Re_{kl}$ with $k,l \in \mathbb{Z}$ are the relevant, dynamical ones and one can choose the measure $D A_1$ to be 
\begin{equation*}
\prod_{k,l \in \ZH} d \Im_{kl} d \Re_{kl} \, ,
\end{equation*}
an infinite product of ordinary Lebesgue measures. The only thing one has to take care of is the appearance of $A_1$ in the gauge fixing term. But if one chooses the gauge fixing $h$ appropriately, for instance linear in $A_1$, one can also achieve  that the parameters $\l_k^i$ and $\mu_l^i$ assemble to $\Re_{kl}$ and $\Im_{kl}$ as for the Chern-Simons action. Hence we obtain a sensible measure for the path integral.
\end{rem}

There are two properties of the Chern-Simons action for the quantum group $SU_q(2)$ which should be emphasized. Firstly, due to theorem \ref{chern_simons_q}, the linear part of our definition of the Chern-Simons action, resulting from the $\phi_1$-cochain, does not vanish. This makes the difference in comparison to the noncommutative $3$-torus or spectral triples coming from a $3$-dimensional spin manifold, where the linear part vanishes identically due to proposition 3.2 in \cite{pfante_torus} or corollary \ref{cor_3.3} respectively. \\
Secondly, the linear part shifts the critical points, the 1-forms where the Chern-Simons action is extremal. This is important for the following reason. The Chern-Simons action on equivalence classes of hermitian $1$-forms $A$, modulo gauge transformations, is only well-defined up to an additional additive constant $2 \pi m$, for an integer $m$. Hence we can not perform Wick rotation in order to replace the integrand $\exp ( 2 \pi i S_{CS}(A) ) $ of the path integral by $\exp ( 2 \pi S_{CS}(A) ) $, because the last term depends on the special choice of the representative $A$. This problem emerges also in the classical, commutative case investigated by Witten in \cite{witten}. Witten circumvents these problems by the use of techniques which are not available in noncommutative geometry: eta invariants, frames, etc. Even though we are able to make sense of the path integral for the quantum group $SU_q(2)$ , it seems to be impossible to reproduce Witten's calculations in our noncommutative framework. Therefore, we do not want to compute the path integral itself but only its $2$-loops, i. e. we calculate the coefficient of $k^{-1}$ in the Taylor expansion series of the path integral $Z(k)$ in the coupling constant $k^{-1}$. This was done, for instance, for the noncommutative $3$-torus in \cite{pfante_torus}.\\
In order to compute the path integral by means of a loop expansion we have to perform this expansion at an extremal point of the Chern-Simons action $S_{CS}$. Heuristically, this can be understood as follows. We compute the integral of $\exp(iS_{CS}(A))$ over all hermitian 1-forms $A$ up to gauge equivalence. For a 1-form $A_0$ where the first variation of the action $S_{CS}$ does not vanish, the integrand $\exp(iS_{CS}(A))$ is oscillatory, so the contribution of neighbouring 1-forms $A$ will tend to cancel out. But for connection $1$-forms where the Chern-Simons action is extremal, the first variation of the action vanishes so that nearby 1-forms give constructive rather than destructive interferences. In other words, the major contribution to the path integral comes from 1-forms where the action is stationary. In the classical, commutative case, where no additional linear terms emerge in the formula of the action functional, the Chern-Simons action is extremal iff the curvature of the connection 1-form is zero (see proposition 3.1 \cite{freed} for a proof of this statement). Connection 1-forms with vanishing curvature are called \textit{flat}.\\
For $SU_q(2)$ things do not work so easily any longer because the appearance of the non vanishing linear part $\phi_1$ forces a shift of the critical points to be different from flat ones. In order to compute the $2$-loops of path integral for $SU_q(2)$ we must find the shifted extremal points. This seems rather difficult. If one takes a closer look at the formulas in theorem \ref{chern_simons_q} for the Chern-Simons action, the coupling of the indices of the independent variables $\Re_{kl}$ and $\Im_{kl}$ by means of the $\d_{\ldots}$-terms in the expression of the $\phi_3$-cochain, and the additional linear term induced by the cochain $\phi_1$ is remarkable. These two obstacles make a correct choice of the parameters $\Im^0_{kl}$ and $\Re^0_{kl}$, such that the Chern-Simons action becomes extremal, rather difficult. At this point, numerical methods should be applied. \\

Finally, we would like to mention that this is not the first time a classical action is expanded by a linear term in the noncommutative setting. In \cite{wulkenhaar_victor} V. Gayal and R. Wulkenhaar investigate the Yang-Mills action for a triple constructed on the noncommutative d-dimensional Moyal space. Moreover, in \cite{wulkenhaar_victor} the authors discovered an additional linear part of the Yang-Mills action for this noncommutative space. With this linear part similar problems arises as in our Chern-Simons setting, due to the shift of the critical points.

\section{Chern-Simons action on noncommutative spaces -- a topological invariant?}

For ordinary $3$-dimensional manifolds one obtains the Chern-Simons action by integration of a $3$-form, i. e. a volume form. Hence the result does not depend on the metric of the manifold. For noncommutative spaces it is not as easy to establish an analogous result. Topological invariance in this case would mean that the action of definition \ref{CS} depends only on the underlying $C^\ast$-algebra $A$ of the spectral triple $(\A, \H, \D)$, but not on the differential structure incorporated into the spectral triple by the Dirac operator $\D$, Hilbert space $\H$, and the pre-$C^\ast$-algebra $\A \subset A$ which contains, due to regularity, the \lgans smooth\rgans $\,$ elements of $A$. But exactly these additional ingredients are involved in the definition of the cochains $\phi_3$ and $\phi_1$. Hence it is not obvious why the action should be a topological invariant as in the classical case.\\

In fact, this does not hold true. The different spectral triples in \cite{pal} and \cite{Landi_1} for the quantum group $SU_q(2)$ provide a counterexample. In many ways the triple $(C^\infty(SU_q(2)), \H ,\D)$ constructed in \cite{Landi_1} is very similar to the one in \cite{pal}. There is essentially one important difference, the definition of the Hilbert space $\H$. The Hilbert space in \cite{Landi_1} is a doubled version of the one in \cite{pal}. More precisely, the Hilbert space of spinors $\H$ has an orthonormal basis labelled as follows. For each $j = 0,\half,1,\dots$, we abbreviate $j^+ = j + \half$ and $j^- = j - \half$. The orthonormal basis consists of vectors $\ket{j\mu n\up}$ for $j = 0,\half,1,\dots$, $\mu = -j,\dots,j$ and $n = -j^+,\dots,j^+$; together with $\ket{j\mu n\dn}$ for $j=\half,1,\dots$, $\mu = -j,\ldots,j$ and $n = -j^-,\dots,j^-$. We adopt a vector notation by juxtaposing the pair of spinors
\begin{equation*}
\kett{j\mu n} := \begin{pmatrix} \ket{j\mu n\up} \\[2\jot]
\ket{j\mu n\dn} \end{pmatrix},
\end{equation*} 
and with the convention that the lower component is zero when $n = \pm(j + \half)$ or $j = 0$. In this way, we get a decomposition $\H = \H^\up \oplus \H^\dn$ into subspaces spanned by the ``up'' and ``down'' kets respectively. \\
The Dirac operator $\D$ is diagonal in the given orthonormal basis of $\H$, and is a selfadjoint operator of the form
\begin{equation*}
\D \kett{j\mu n} = \begin{pmatrix} 2 j + \sesq & 0 \\
0 & -2 j - \half \end{pmatrix} \kett{j\mu n}.
\end{equation*} 
The spectrum of this Dirac operator coincides with the one of the classical Dirac operator of the sphere $S^3$ equipped with the standard metric. We let $\D = F \, \AD$ be the polar decomposition of $\D$ where $ \AD := \sqrt{\D ^2} $ and $F = \sgn \D$. Explicitly, we see that
\begin{equation*}
F\kett{j\mu n}
= \begin{pmatrix} 1 & 0 \\ 0 & -1 \end{pmatrix} \kett{j\mu n},
\qquad
\mid \D \mid \,\kett{j\mu n} = \begin{pmatrix} 2j + \sesq & 0 \\
0 & 2j + \half \end{pmatrix} \kett{j\mu n}.
\end{equation*}
Clearly, $P^\up := \half(1 + F)$ and $P^\dn := \half(1 - F) = 1 - P^\up$ are the orthogonal projectors whose range spaces are $\H^\up$ and $\H^\dn$ respectively. \\
The precise form of the representation of the generators $\a$ and $\b$ is not necessary for our purposes. Hence we skip the definitions and refer the reader to \cite{Landi_1}. \\

In \cite{Landi_2} the authors proceed by constructing a cosphere bundle $C^{\infty}(S^\ast_q)$ for $(C^\infty(SU_q(2)), \H, \D)$ in an analogous way as in \cite{connes_sphere}, and end up with a similar result for the local index formula. The maps $\rho, \tau_1, \tau_0^\up, \tau_0^\dn$ and the grading are defined in a similar way as in subsection 4.2. Let us state theorem 4.1 in \cite{Landi_2}.

\begin{thm} \label{thm_6.1}
The dimension spectrum of the spectral triple $(C^\infty(SU_q(2)),\H,\D)$ is
simple and given by $\lbrace 1,2,3 \rbrace$; the corresponding residues are
\begin{align*}
\ncint T \AD^{-3} &= 2 (\tau_1 \ot \tau_1) \left( r \rho(T)^0 \right)\, , \\
\ncint T \AD^{-2} &= \left( \tau_1 \ot (\tau_0^\up + \tau_0^\dn) + (\tau_0^\up + \tau_0^\dn) \ot \tau_1 \right) \left(r\rho(T)^0\right) \, , \\
\ncint T \AD^{-1} &= (\tau_0^\up \ot \tau_0^\dn + \tau_0^\dn \ot \tau_0^\up) \left( r \rho(T)^0 \right) \, ,
\end{align*}
\begin{align*}
\ncint P^\up T \AD^{-3} &= (\tau_1 \ot \tau_1) \left(r\rho(T)^0\right) \, , \\
\ncint P^\up T \AD^{-2} &= \left(\tau_1 \ot \tau_0^\dn + \tau_0^\up \ot \tau_1 \right) \left(r\rho(T)^0\right) \, , \\
\ncint P^\up T \AD^{-1} &= (\tau_0^\up \ot \tau_0^\dn) \left(r\rho(T)^0\right) \, ,
\end{align*}
with $T \in \Br_q$. $\rho(T)^0$ denotes the zero graded part of $\rho(T)$ in the cosphere bundle $C^{\infty}(S^\ast_q)$.
\end{thm}

Knowing all these residues we can compute the cyclic cocycle $(\phi_3, \phi_1)$ of the local index theorem in general, and the chochain $\phi_3$ in particular. Due to the simplicity of the dimension spectrum the cochain $\phi_3$ reduces to the formula 
\begin{equation*}
\phi_3(a^0da^1da^2da^3) = \frac{1}{12} \ncint a^0\, [D,a^1]\,[D,a^2]\,[D,a^3] \AD ^{-3} \, .
\end{equation*}

\begin{cor}
Let $(C^\infty(SU_q(2)), \H, \D)$ be the spectral triple for $SU_q(2)$ constructed in \cite{Landi_1}. Then the cochain $\phi_3$ vanishes identically.
\end{cor}

\begin{proof}
In \cite{Landi_2} it was proven that $\lcom F, x \rcom$ is a trace class operator for all $x \in \Psi^{0}(\A)$, where $\Psi^{0}(\A)$ denotes the algebra which is generated by $\delta^{k}(\A)$ and $\delta^{k}(\lcom \D, \A \rcom)$ for $k \geq 0$, with $\d(\cdot) = \lcom \AD, \, \cdot \, \rcom$. Since the space of trace class operators is an ideal in $\B(\H)$ and $\ncint T = 0$ for any trace class operator $T$, we can rewrite the cochain $\phi_3$ in the following form: 
\begin{align*}
\phi_3(a_0da_1da_2da_3) =& \frac{1}{12} \ncint a_0\,[D,a_1]\,[D,a_2]\,[D,a_3]\,|D|^{-3} \\
=& \dfrac{1}{12} \ncint F a_0\,[|D|,a_1]\,[|D|,a_2]\,[|D|,a_3]\,|D|^{-3}.
\end{align*} 
We have the identity $F = 2 P^\up -1 $ and with theorem \ref{thm_6.1} we obtain 
\begin{align*}
\phi_3(a_0da_1da_2da_3) =& \quad \dfrac{1}{12} \ncint F a_0\,[|D|,a_1]\,[|D|,a_2]\,[|D|,a_3]\,|D|^{-3} \\
=& \quad \dfrac{1}{6} \ncint  P^\up a_0\,[|D|,a_1]\,[|D|,a_2] \,[|D|,a_3]\,|D|^{-3} \\
&- \dfrac{1}{12} \ncint  a_0\,[|D|,a_1]\,[|D|,a_2]\,[|D|,a_3] \,|D|^{-3}\\
=& \quad \dfrac{1}{6}(\tau_1 \ot \tau_1) \left(r\rho \left(a_0\,[|D|,a_1]\,[|D|,a_2]\,[|D|,a_3] \right)^0\right)  \\
& - \dfrac{1}{12} \left( 2 (\tau_1 \ot \tau_1) \left(r\rho \left( a_0\,[|D|,a_1]\,[|D|,a_2]\,[|D|,a_3] \right)^0\right) \right)\\
 =& \quad 0
\end{align*} 
\end{proof}

Since the cochain $\phi_3$ vanishes the Chern-Simons action for the triple in \cite{Landi_1} reduces to the $\phi_1$-term, i. e. for any matrix of 1-forms $ A \in \Omega^1(M_N(C^\infty(SU_q(2))))$, we obtain $S_{CS}(A) = -2 \pi k \phi_1(A)$. Hence the action is linear in $A$ which is certainly not true for the action computed in theorem \ref{chern_simons_q}.

\begin{thm}
The Chern-Simons action of definition \ref{CS} depends not only on the underlying $C^\ast$-algebra $A$ but also on the particular spectral triple $(\A, \H, \D)$. Hence it is not a topological invariant.
\end{thm}

\end{document}